\documentclass[10pt,a4paper]{article}
\usepackage[T1]{fontenc}
\usepackage[english]{babel}
\usepackage{amsmath}
\usepackage{amsfonts}
\usepackage{amssymb}
\usepackage{amsthm}
\usepackage{titlesec}
\usepackage{indentfirst}
\usepackage{color}
\usepackage{enumerate}

\newtheorem{lemme}{Lemma}

\newtheorem*{prop}{Proposition}
\newtheorem*{theo}{Theorem}
\newtheorem{theoreme}{Theorem}
\newtheorem{cor}{Corollary}
\newtheorem{proposition}{Proposition}

\theoremstyle{definition}
\newtheorem*{defi}{Definition}
\newtheorem{defin}{Definition}
\newtheorem*{rem}{Remark}
\newtheorem*{rems}{Remarks}

\newtheorem*{claim}{Claim}
\newtheorem*{ex}{Example}
\newtheorem*{exs}{Examples}
\newtheorem*{qu}{Questions}
\newtheorem*{nota}{Notation}

\title{Composition operators on weighted Bergman spaces of Dirichlet series}
\date{}

\newcommand\e{{\rm e}}

\newcommand{\biindice}[3]% 
{%

\begin{array}[t]{c}
{\displaystyle #1}\\
{\scriptstyle #2}\\
{\scriptstyle #3}
\end{array}

}

\titleformat{\section}[block]
{\normalfont\large\bfseries\filcenter}
{\thesection}
{1em}
{}

\titleformat{\subsection}[block]
{\normalfont\bfseries\filcenter}
{\thesubsection}
{1em}
{}

\renewcommand{\thesubsection}{\Alph{subsection}}

\begin{document}
\begin{center}
{\Large \bfseries Composition operators on weighted Bergman spaces of Dirichlet series}
\vspace{5pt}

{ \textit{ Maxime Bailleul }

 				   \par\vspace*{20pt}} 
\end{center}

\begin{abstract} We study boundedness and compactness of composition operators on weighted Bergman spaces of Dirichlet series. Particularly, we obtain in some specific cases, upper and lower bounds of the essential norm of these operators and a criterion of compactness on classicals weighted Bergman spaces. Moreover, a sufficient condition of compactness is obtained using the notion of Carleson's measure.
\end{abstract}

\begin{center} {\footnotesize {\bfseries Mathematics Subject Classification.}

 Primary: 47B33 - Secondary: 46E15.} \end{center}
 
\begin{center} {\footnotesize {\bfseries Keywords.}

 weighted Bergman spaces - Carleson measure - composition operator - Dirichlet series - essential norm - generalized Nevanlinna counting function} \end{center} 

\section{Introduction}
In \cite{hedenmalm1995hilbert}, the authors defined the Hardy space $ \mathcal{H}^2$ of Dirichlet series with square-summable coefficients. Thanks to the Cauchy-Schwarz inequality, it is easy to see that $\mathcal{H}^2$ is a  space of analytic functions on $ \mathbb{C}_{ \frac{1}{2}}:= \lbrace s \in \mathbb{C}, \, \Re(s) > \frac{1}{2} \rbrace$ and that this domain is maximal. F. Bayart introduced in \cite{bayart2002hardy} the more general class of Hardy spaces of Dirichlet series $ \mathcal{H}^p$ ($1 \leq p < + \infty$). 

%In another direction, McCarthy defined in \cite{mccarthy2004hilbert} some weighted Hilbert spaces of two types: Bergman-like spaces and Dirichlet-like spaces. 
%In \cite{bailleullefevre2013} the authors defined two classes of Bergman spaces of Dirichlet series denoted by $\mathcal{A}^p$ and $ \mathcal{B}^p$.
In \cite{gordon1999composition}, the bounded composition operators on $ \mathcal{H}^2$ are characterized, that is to say the holomorphic functions $\Phi: \mathbb{C}_{ \frac{1}{2}} \rightarrow \mathbb{C}_{ \frac{1}{2}}$ such that for any $f \in \mathcal{H}^2$, $f \circ \Phi \in \mathcal{H}^2$. We denote by $\mathcal{D}$ the class of functions $f$ which admit representation by a convergent Dirichlet series in some half-plane and for $\theta \in \mathbb{R}$, $\mathbb{C}_{\theta}$ will be the following half-plane $\lbrace s \in \mathbb{C}, \, \Re(s)> \theta \rbrace$. We shall denote $\mathbb{C}_+$ instead of $\mathbb{C}_{0}$.

\begin{theo}[\cite{gordon1999composition},\cite{queffelec2013approximation}]
The function $\Phi$ determines a bounded composition operator on $\mathcal{H}^2$ if and only:
\[\Phi(s) = c_0 s + \varphi(s)\]
 where $c_0$ is a nonnegative integer, $\varphi \in \mathcal{D}$ and $ \varphi$ converges uniformly in $ \mathbb{C}_{\varepsilon}$ for every $ \varepsilon>0$ and has the following properties:
 
a) If $c_0 \geq 1$, $ \varphi(\mathbb{C}_+) \subset \mathbb{C}_+$.

b) If $c_0=0$, $ \varphi( \mathbb{C}_+) \subset \mathbb{C}_{1/2}$.

Moreover, when $c_0 \geq 1$, $C_{ \Phi}$ is a contraction.
\end{theo}

This theorem is a strengthening of the original result (see \cite{queffelec2013approximation},Th3.1). In \cite{bayart2002hardy}, F. Bayart proved that this result is true on the space $\mathcal{H}^p$ when $p \geq 1$ and $c_0 \neq 0$. When $c_0=0$ the condition is necessary for every $p \geq 1$ and sufficient if $p$ is even. The goal of this paper is to study composition operators on the weighted Bergman $\mathcal{A}_{\mu}^p$ spaces of Dirichlet series defined in \cite{bailleullefevre2013}. In order to define them, we need to recall the Bohr's point of view principle and the definition of the Hardy spaces $\mathcal{H}^p$.

Let ${n\geq 2}$ be an integer, we shall denote $\e_n$ the function defined by \linebreak $\e_n(z)=n^{-z}$. The integer $n$ can be written (in a unique way) as a product of prime numbers $n=p_1^{ \alpha_1} \cdot \cdot \, \, p_k^{ \alpha_k}$ where $p_1=2
, \, p_2=3 $ etc \dots If $s$ is a complex number, we consider $z=(p_1^{-s}, \, p_2^{-s}, \dots )$. let $f$ be a Dirichlet series of the following form
\begin{equation}
f= \sum_{n=1}^{+ \infty} a_n \e_n,
\end{equation}
then
\[ f(s) = \sum_{n=1}^{ + \infty } a_n (p_1^{-s})^{ \alpha_1} \cdot \cdot \, \, (p_k^{-s})^{ \alpha_k} = \sum_{n=1}^{ + \infty } a_n \, z_1^{ \alpha_1} \cdot \cdot \, \, z_k^{ \alpha_k}.  \]
So we can see a Dirichlet series as a Fourier series on the infinite-dimensional polytorus $ \mathbb{T}^{ \infty}= \lbrace(z_1,z_2, \cdots), \, \vert z_i \vert=1, \, \forall i \geq 1 \rbrace$. We shall denote this Fourier series $D(f)$. 

Let us fix now $p\ge1$, the space $H^p( \mathbb{T}^{ \infty})$ is the closure of the set of analytic polynomials with respect to the norm of $L^p( \mathbb{T}^{ \infty}, \, m)$ where $m$ is the normalized Haar measure on $\mathbb{T}^{\infty}$ (see \cite{cole1986representing} for more details). Let $f$ be a Dirichlet polynomial, by the Bohr's point of view $D(f)$ is an analytic polynomial on $ \mathbb{T}^{ \infty}$ and by definition $ \Vert f \Vert_{ \mathcal{H}^p}:= \Vert D(f) \Vert_{ H^p( \mathbb{T}^{ \infty})}$. The space $ \mathcal{H}^p$ is defined by taking the closure of Dirichlet polynomials with respect to this norm. Consequently $ \mathcal{H}^p$ and $ H^p( \mathbb{T}^{ \infty})$ are isometrically isomorphic. \\

For $ \sigma>0$, $f_{ \sigma}$ will be the translate of $f$ by $\sigma$, \textit{ie} ${f_{ \sigma}(s):= f ( \sigma+s)}$. We shall denote by $\mathcal{P}$ the set of Dirichlet polynomials.

Let $p \geq 1$, $P \in \mathcal{P}$ and let $ \mu$ be a probability measure on $(0, + \infty)$ such that $0 \in Supp( \mu)$. Then 
\[ \Vert P \Vert_{ \mathcal{A}_{ \mu}^p}:= \bigg{(} \int_{0}^{+ \infty} \Vert P_{ \sigma} \Vert_{ \mathcal{H}^p}^p \, d \mu( \sigma) \bigg{)}^{1/p}. \]
$\mathcal{A}_{\mu}^p$ is the completion of $\mathcal{P}$ with respect to this norm. When $ d\mu( \sigma) = 2 e^{-2 \sigma} \, d \sigma$, these spaces are simply denoted by $ \mathcal{A}^p$. It is shown in \cite{bailleullefevre2013} that they are spaces of convergent Dirichet series on $\mathbb{C}_{1/2}$. \\

In section $2$, we recall some facts about Dirichlet series and precise some notations.

In section $3$, we study the boundedness of $C_{ \Phi}$ on $\mathcal{A}_{\mu}^p$. For $\Phi(s)=c_0s + \varphi(s)$ with $c_0 \geq 1$ and $\varphi \in \mathcal{D}$, $\Phi$ induces a bounded composition operator from $ \mathcal{A}_{\mu}^p$ to $ \mathcal{A}_{\mu}^p$ if and only if $\varphi$ converges uniformly on $ \mathbb{C}_{\varepsilon}$ for every $ \varepsilon>0$ and $ \varphi( \mathbb{C}_+) \subset \mathbb{C}_+$. When $c_0=0$, we give sufficient and necessary  conditions.

In section $4$, we define generalized Nevanlinna counting function in order to obtain some estimates of the essential norm of $C_{ \Phi}$ on $ \mathcal{H}^2$ and $ \mathcal{A}_{ \mu}^2$. Particularly we obtain a criterion of compactness of $C_{\Phi}$ in some specific cases.

In section $5$, we obtain a majorization of the generalized Nevanlinna counting function with help of the associated Carleson function (see definition in section $5$) and then we get other sufficient conditions of compactness of $C_{\Phi}$. 

\section{Background material}
Let $f$ be a Dirichlet series of form $(1)$. We do not recall the definition of abscissa of simple (resp. absolute) convergence denoted by $\sigma_c$ (resp. $\sigma_a$), see \cite{queffelecBook} or \cite{tenenbaum1995introductiona} for more details. We shall need the two other following abscissas:
\[ \left. \begin{array}{ccl}
%\sigma_c(f) & = &\inf \lbrace a \, | \hbox{ The series }\,(1) \hbox{ is convergent for } \Re(s)>a \rbrace   \\
%            & = & \hbox{abscissa of simple convergence of } f.   \\
%            &   &         \\
%\sigma_a(f) & = &\inf \lbrace a \, | \hbox{ The series }\,(1) \hbox{ is absolutely convergent for }\Re(s)>a
%\rbrace \\
%            & = & \hbox{abscissa of absolute convergence of } f.   \\
%            &   &         \\
\sigma_u(f) & = &\inf \lbrace a \, | \hbox{ The series }\,(1) \hbox{ is uniformly convergent for } \Re(s)>a \rbrace   \\
            & = & \hbox{abscissa of uniform convergence of } f.   \\
            &   &                                                 \\   
\sigma_b(f) & = &\inf \lbrace a \, | \hbox{ the function } f \hbox{ has an analytic, bounded extension for } \Re(s)>a \rbrace   \\
            & = & \hbox{abscissa of boundedness of } f.  
\end{array}\right. \]

It is easy to see that $\sigma_c(f) \leq \sigma_u(f) \leq \sigma_a(f)$. An important result is that $\sigma_u(f)$ and $\sigma_b(f)$ coincide: this is the Bohr's theorem (see \cite{bohr1913uber}), it is really important for the study of $ \mathcal{H}^{ \infty}$, the algebra of bounded Dirichlet series on the right half-plane $\mathbb{C}_+$ (see \cite{maurizi2010some}). We shall denote by $ \Vert \,\cdot \,\Vert_{ \infty}$ the norm on this space:
\[ \Vert f \Vert_{ \infty}:= \sup_{ \Re(s)>0} \vert f(s) \vert. \]

%Let us recall now the Bohr's point of view on Dirichlet series: let ${n\geq 2}$ be an integer, we can write it (uniquely) as a product of prime numbers \linebreak $n=p_1^{ \alpha_1} \cdot \cdot \, \, p_k^{ \alpha_k}$ where $p_1=2
%, \, p_2=3 $ etc \dots If $s$ is a complex number and if we consider $z=(p_1^{-s}, \, p_2^{-s}, \dots )$, then we have by $(1)$ 
%\[ f(s) = \sum_{n=1}^{ + \infty } a_n (p_1^{-s})^{ \alpha_1} \cdot \cdot \, \, (p_k^{-s})^{ \alpha_k} = \sum_{n=1}^{ + \infty } a_n \, z_1^{ \alpha_1} \cdot \cdot \, \, z_k^{ \alpha_k}.  \]
%So we can see a Dirichlet series as a Fourier series on the infinite-dimensional polytorus $ \mathbb{T}^{ \infty}$. We shall denote this Fourier series $D(f)$. This correspondence is not just formal. For instance, let $ \mathbb{P}$ be the set of primes numbers, Bohr proved the next result.
%
%\begin{theo}[\cite{bohr1913uber}]
%Let $f$ be a Dirichlet series of the form $(1)$. Then  
%\[ \sum_{ p \in \mathbb{P}} \vert a_p \vert \leq \Vert f \Vert_{\infty}. \]
%\end{theo}

The infinite-dimensional polytorus $ \mathbb{T }^{ \infty}$ can be identified with the group of complex-valued characters $ \chi$ on the positive integers which satisfy the following properties  
\[\left\lbrace
\begin{array}{ll}
\vert \chi(n) \vert=1 &\quad \forall n \geq 1,  \\
 \chi(nm) = \chi(n) \, \chi(m) &\quad \forall n, \, m \geq 1. \end{array}   \right. \]
To obtain this identification for $ \chi = ( \chi_1, \, \chi_2, \dots) \in \mathbb{T}^{ \infty}$, it suffices to define $\chi$ on the prime numbers by $ \chi(p_i)= \chi_i$ and use multiplicativity. \\

Let $f$ be a Dirichlet series absolutely convergent in a half-plane. For any sequences $( \tau_n) \subset \mathbb{R}$, we consider vertical translations of $f$:
\[ (f_{ \tau_n}(s)):=(f(s + i \tau_n)).  \]
By Montel's theorem, this sequence is a normal family in the half-plane of absolute convergence of $f$ and so there exists a convergent subsequence, say $(f_{\tau_{n_k}})$, such that $f_{\tau_{n_k}}$ converges uniformly on compact subsets of the domain of absolute convergence of $f$ to a limit function  $\tilde{f}$. We say that  $\tilde{f}$ is a vertical limit of $f$. 

\begin{prop}[\cite{hedenmalm1995hilbert}]
Let $f$ be a Dirichlet series of the form (1), absolutely convergent in a half-plane. The vertical limit functions of $f$ are exactly the functions of the form $f_{ \chi} = \displaystyle{\sum_{n \geq 1}} a_n \chi(n) \, \e_n$ where $\chi \in \mathbb{T}^{ \infty}$.
\end{prop}

When $f \in \mathcal{A}_{\mu}^p$, the vertical limits have good properties:

\begin{prop}[\cite{bailleullefevre2013}]
Let $p \geq 1$ and $f \in \mathcal{A}_{\mu}^p$. For almost all $\chi \in \mathbb{T}^{ \infty}$, $f_{ \chi}$ converges on $\mathbb{C}_+$.
\end{prop}

Let $\Phi(s)=c_0s+ \varphi(s)$ where $c_0$ is a nonnegative integer and $\varphi \in \mathcal{D}$. For $\tau \in \mathbb{R}$, we define $\Phi_{\tau}$ by $\Phi_{\tau}(s)= c_0 s + \varphi(s + i \tau)$, then if $\chi \in \mathbb{T}^{\infty}$ the function $\Phi_{\chi}$ defined by 
\[ \Phi_{ \chi}(s) = c_0s + \varphi_{ \chi}(s) \]
is a vertical limit of the functions $(\Phi_{\tau})_{\tau \in \mathbb{R}}$. \\

In section $4$, we shall consider the Hilbert spaces of Dirichlet series $\mathcal{H}^2$ and $\mathcal{A}_{\mu}^2$ where $\mu$ is a probability measure on $(0, + \infty)$ such that $0 \in Supp( \mu)$. Moreover we shall assume that $d\mu( \sigma) = h( \sigma) d \sigma$ where $h$ is a nonnegative function such that $ \displaystyle{\Vert h \Vert_{ L^1((0, + \infty))}}=1$. 
We see easily that for an element of $ \mathcal{A}_{\mu}^2$ with form $(1)$:
\[ \Vert f \Vert_{ \mathcal{A}_{\mu}^2} =  \bigg{(} \sum_{n=1}^{ + \infty} \vert a_n \vert^2 w_h(n) \bigg{)}^{1/2} \]
where
\[ w_h(n) = \int_{0}^{ + \infty} n^{-2 \sigma} h( \sigma) d \sigma \quad \quad \quad \quad \forall n \geq 1. \]

\begin{exs}
We follow notations from \cite{bailleullefevre2013}: let $\alpha>-1$, we denote $\mu_{\alpha}$ the probability measure defined on $(0, + \infty)$ by
\[ d\mu_{\alpha}( \sigma) = \frac{2^{\alpha+1}}{\Gamma(\alpha+1)} \sigma^{\alpha} \exp(-2 \sigma) d \sigma.\]
We denote this space $\mathcal{A}_{\alpha}^2$ instead of $ \mathcal{A}_{\mu_{\alpha}}^2$ and the corresponding weight by $(w_n^{\alpha})$. Then for every $n \geq 1$,
\[ w_{n}^{\alpha} = \frac{1}{(\log(n)+1)^{\alpha+1}}. \]
\end{exs}

We shall denote $(\e_{n}^{\mu})_{n \geq 1}$ the orthonormal basis of $\mathcal{A}_{\mu}^2$ defined by
\[ \e_n^{\mu}:= \frac{\e_n}{\sqrt{w_h(n)}}. \]

\section{Boundedness of composition operators}
In the sequel $\mu$ will be a probability measure on $(0, + \infty)$ such that $0 \in Supp( \mu)$ (as in \cite{bailleullefevre2013}).

\begin{theoreme}{\label{Continuitécoplusgrandque1}} Let $ \Phi: \mathbb{C}_{ \frac{1}{2}} \rightarrow \mathbb{C}_{ \frac{1}{2}}$ be an analytic function of the form $ \Phi(s) =c_0s + \varphi(s)$ where $c_0 \geq 1$ and $ \varphi \in \mathcal{D}$. Then $C_{\Phi}$ is bounded on $\mathcal{A}_{\mu}^p$ if and only $ \varphi$ converges uniformly in $ \mathbb{C}_{\varepsilon}$ for every $ \varepsilon>0$ and $ \varphi( \mathbb{C}_+) \subset \mathbb{C}_+$ and in this case $C_{\Phi}$ is a contraction.
\end{theoreme}

\begin{rem}
The condition $\Phi(s)=c_0s + \varphi(s)$ where $c_0$ is a nonnegative integer is an arithmetical condition: $\Phi$ has this form if and only if $P \circ \Phi \in \mathcal{D}$ for every $P \in \mathcal{P}$ (see \cite{gordon1999composition},Th.A).
\end{rem}

%We need the following lemma.
%
%\begin{lem}{\cite{gordon1999composition}}
%Let $\theta, \, \mu \in \mathbb{R}$. Assume that $ \Phi: \mathbb{C}_{ \theta} \rightarrow \mathbb{C}_{ \mu}$ is an analytic function of the form $ \Phi(s) =c_0 s + \varphi(s)$ where $c_0 \in \mathbb{N}$ and $ \varphi \in \mathcal{D}$. Then if $ \varphi$ is constant, this constant is in $ \overline{\mathbb{C}_{ \mu-c_0 \theta}}$. Either, $ \varphi$ extends to an analytic function on $ \mathbb{C}_{ \theta}$ to $\mathbb{C}_{ \mu-c_0 \theta}$.
%\end{lem}

\begin{proof}
$\hbox{  }$

\noindent $\rhd$ Assume that $\varphi$ converges uniformly on $ \mathbb{C}_{\varepsilon}$ for every $\varepsilon>0$ and that \linebreak $ \varphi( \mathbb{C}_+) \subset \mathbb{C}_+$. Let $P$ be a Dirichlet polynomial, $P \circ \Phi$ is then a Dirichlet series which is bounded on $\mathbb{C}_+$ (because $\Phi( \mathbb{C}_+) \subset \mathbb{C}_+)$ and so $P \circ \Phi$ belongs to $\mathcal{H}^{ \infty}$ and consequently to  $\mathcal{H}^p$. By \cite{bailleullefevre2013},Th.6(ii) we know that
\[ \Vert P \circ \Phi \Vert_{A_{\mu}^p}^p = \int_{0}^{ + \infty} {\Vert (P \circ \Phi)_{ \sigma} \Vert}_{ \mathcal{H}^p}^p  \, d \mu( \sigma). \]

Now, the key point is to use the boundedness of the composition operators on $ \mathcal{H}^p$. We remark that for each $\sigma>0$,
\[ (P \circ \Phi)_{ \sigma} (it) = P ( \Phi ( \sigma+it)) = P_{ \sigma} ( \Phi( \sigma +it) -  \sigma). \]

Let $\widetilde{\Phi}$ defined on $ \mathbb{C}_{+}$ by  $\widetilde{\Phi}(s) = \Phi( \sigma+s) -\sigma$. We claim that $\widetilde{\Phi}$ verifies the condition of the boundedness of $C_{\widetilde{\Phi}}$ on $ \mathcal{H}^p$: it can be written $\widetilde{\Phi}(s)= c_0s + \widetilde{\varphi}(s)$ where $ \widetilde{\varphi}$ converges uniformly on $\mathbb{C}_{\varepsilon}$ for every $\varepsilon>0$. Then it suffices to check that $ \widetilde{\varphi}( \mathbb{C}_+) \subset \mathbb{C}_+$ but $\widetilde{ \varphi}(s) = \varphi(s + \sigma ) +(c_0-1)\sigma$ for every $s \in \mathbb{C}_+$ and the conclusion is clear because $c_0 \geq 1$ and $ \varphi( \mathbb{C}_+) \subset \mathbb{C}_+$.

%Let $s$ such that $ \Re(s) > 0$. Then $\Re( \sigma+c +s) > \sigma + c$ and 
%\[ \Re( \Phi( \sigma +c +s)) = c_0 \Re( \sigma +c + s) + \Re( \varphi( \sigma +c +s)). \]
%By the previous lemma we know that $\Re( \varphi( \sigma +c +s))  \geq 0$.
%\[ \Rightarrow \Re( \Phi( \sigma +c +s)) > c_0 ( \sigma + c ) \geq \sigma+c \]
%because $c_0 \geq 1$. Finally $\Re( \hat{\Phi}(s)) > 0$ for every $s \in \mathbb{C}_+$. 
Now we apply the result on the boundedness on $ \mathcal{H}^p$. Since $c_0 \geq 1$, $C_{ \widetilde{\Phi}}$ is a contraction on $\mathcal{H}^p$ and we obtain ($\sigma$ is fixed):
\[ {\Vert (P \circ \Phi)_{ \sigma} \Vert}_{ \mathcal{H}^p}^p = {\Vert P_{ \sigma} (\widetilde{\Phi} ) \Vert}_{ \mathcal{H}^p}^p \leq {\Vert P_{ \sigma} \Vert}_{ \mathcal{H}^p}^p. \]
This is true for every $ \sigma > 0$, so
\[ \int_{0}^{ + \infty} {\Vert (P \circ \Phi)_{ \sigma} \Vert}_{ \mathcal{H}^p}^p  \, d \mu( \sigma) \leq    \int_{0}^{ + \infty} {\Vert P_{ \sigma} \Vert}_{ \mathcal{H}^p}^p  \, d \mu( \sigma). \]

Hence $\Vert P \circ \Phi \Vert_{ \mathcal{A}_{\mu}^p} \leq \Vert P \Vert_{ \mathcal{A}_{\mu}^p}$. Now, by density of $\mathcal{P}$ in $ \mathcal{A}_{\mu}^p$, $C_{\Phi}: \mathcal{P} \cap \mathcal{A}_{\mu}^p \rightarrow \mathcal{A}_{\mu}^p$ extends to a bounded operator $T: \mathcal{A}_{\mu}^p \rightarrow \mathcal{A}_{\mu}^p$ such that $T(P)= P \circ \Phi$ for every $P \in \mathcal{P}$. Let $f \in \mathcal{A}_{\mu}^p$, there exists a sequence $(P_n) \subset \mathcal{P}$ such that $(P_n)$ converges to $f$ in norm, then by continuity of the point evaluation at $s \in \mathbb{C}_{1/2}$ on $ \mathcal{A}_{\mu}^p$ (see \cite{bailleullefevre2013},Th.1):
\[ \vert T(f)(s) - T(P_n)(s) \vert \leq \Vert \delta_s \Vert_{(\mathcal{A}_{\mu}^p)^*} \Vert T(f) - T(P_n) \Vert_{ \mathcal{A}_{\mu}^p} \]
\[ \leq \Vert \delta_s \Vert_{(\mathcal{A}_{\mu}^p)^*} \Vert T \Vert \, \Vert f - P_n \Vert_{ \mathcal{A}_{\mu}^p} \]
and then $T(P_n)(s) = P_n( \Phi(s))$ converges to $T(f)(s)$ when $n$ goes to infinity. But by evaluation at $\Phi(s) \in \mathbb{C}_{1/2}$, $P_n( \Phi(s))$ converges to $f( \Phi(s))$, so $T(f)(s)= f( \Phi(s)) = C_{\Phi}(f)(s)$ and the result is proved. \\

%When $p=1$, it suffices to take the limit of the previous inequality when $p$ goes to $1^+$. \\

$\rhd$ Now assume that $\Phi$ induces a bounded composition operator on $\mathcal{A}_{\mu}^p$. The result follows from ideas similar to the ones of the proof of Th.B from \cite{gordon1999composition}. We only give a sketch of the proof.

Let $f \in \mathcal{A}_{\mu}^p$. For every $\chi \in \mathbb{T}^{ \infty}$, we can show that (see \cite{gordon1999composition}, Prop4.3) for every $s \in \mathbb{C}_{1/2}$,
\[ (f \circ \Phi)_{\chi}(s) = f_{ \chi^{c_0}} \circ \Phi_{\chi}(s).  \quad \quad \quad \quad \quad (\Delta)\]
Since $f$ and $f \circ \Phi$ belong to $\mathcal{A}_{\mu}^p$, $(f \circ \Phi)_{\chi}$ and $f_{\chi^{c_0}}$ extend analytically  on $ \mathbb{C}_{+}$ almost-surely relatively to $\chi$ . As in Prop5.1 from \cite{gordon1999composition}, we can show that $\Phi_{\chi}$ extends analytically on $ \mathbb{C}_+$ for every $ \chi \in \mathbb{T}^{\infty}$.

Now, let $\chi \in \mathbb{T}^{\infty}$ such that $f_{\chi^{c_0}}$, $\Phi_{\chi}$ and $(f \circ \Phi)_{\chi}$ extend on $ \mathbb{C}_+$. 

If $\Phi_{\chi}$ does not map $\mathbb{C}_+$ to $ \mathbb{C}_+$, then there exists $s_0 \in  \mathbb{C}_+$ such that $\Phi_{\chi}(s_0)$ is in the imaginary axis and with an argument of connectedness, we can choose it such that $\Phi_{\chi}^{'}(s_0) \neq 0$. Then, by $(\Delta)$, $f_{\chi^{c_0}}$ has an analytic extension on a small segment on the imaginary axis near $\Phi_{\chi}(s_0)$. We denote $(p_n)_{n \geq 1}$ the sequence of prime numbers. In \cite{bayart2002hardy} it is shown that 
\[ f = \sum_{n \geq 1} \frac{e_{p_n}}{\sqrt{p_n} \log(p_n)} \in \mathcal{H}^p  \]
and that almost surely relatively to $\chi$, $f_{\chi^{c_0}}$ does not extend analytically  to any region larger than $\mathbb{C}_{+}$. But $f \in \mathcal{A}_{\mu}^p$ because $ \mathcal{H}^p \subset \mathcal{A}_{\mu}^p$. Consequently $\Phi_{\chi}$ has to map $\mathbb{C}_+$ to $\mathbb{C}_+$ almost surely relatively to $\chi$ and then by Proposition 4.1 from \cite{gordon1999composition}, $\Phi( \mathbb{C}_+) \subset \mathbb{C}_+$. Finally, it suffices to apply Th3.1  from \cite{queffelec2013approximation} to obtain that $ \varphi( \mathbb{C}_+) \subset \mathbb{C}_+$ and $\varphi$ converges uniformly on $ \mathbb{C}_{\varepsilon}$ for all $\varepsilon>0$.
\end{proof}

\begin{rem}
In the previous proof we used that $\Phi_u -u: \mathbb{C}_+ \rightarrow \mathbb{C}_+$ when $c_0 \geq 1$. We can see this result as an equivalent of the Schwarz's lemma in this framework.
\end{rem}

When $c_0 = 0$, we obtain the following result.

\begin{theoreme}{\label{Continuitécoegal0}} Let $p \geq 1$ and $ \Phi : \mathbb{C}_{ \frac{1}{2}} \rightarrow \mathbb{C}_{ \frac{1}{2}}$ be an analytic function belonging to $\mathcal{D}$. Then
\begin{enumerate}[(i)]
\item
If $C_{\Phi}$ is bounded on $\mathcal{A}_{\mu}^p$ then $ \Phi$ converges uniformly on $ \mathbb{C}_{\varepsilon}$ for every $ \varepsilon>0$ and $ \Phi( \mathbb{C}_+) \subset \mathbb{C}_{1/2}$.
\item
If  $\Phi$ converges uniformly on $ \mathbb{C}_{\varepsilon}$ for every $ \varepsilon>0$ and $ \Phi( \mathbb{C}_+) \subset \mathbb{C}_{1/2+ \eta}$ (where $\eta>0$) then $C_{\Phi}$ is Hilbert-Schmidt on $\mathcal{A}_{\mu}^2$, a fortiori bounded.
\item
If $C_{\Phi}$ is bounded on $\mathcal{A}_{\mu}^2$, then it is bounded on $\mathcal{A}_{\mu}^{2k}$ for any $k \geq 1$. Consequently $C_{\Phi}$ is bounded on $\mathcal{A}_{\mu}^{2k}$ for any $k \geq 1$ under conditions of $(ii)$.
\end{enumerate}
\end{theoreme}

\begin{proof} $\hbox{  }$

$(i)$ As in the case $c_0 \geq 1$, the proof follows from \cite{gordon1999composition}, Th.B and some easy adaptations.

$(ii)$ By Lemma $5$ from \cite{bayart2002hardy} we have:
\[ \Vert n^{- \Phi_{\sigma}} \Vert_{\mathcal{H}^2}^2 = \int_{\mathbb{T}^{\infty}} \int_{0}^{1} \vert n^{-\Phi_{\chi}(\sigma+it)} \vert^2 dt dm(\chi) \quad \quad \quad \quad \forall n \geq 1.\]
Now $\Phi : \mathbb{C}_+ \rightarrow \mathbb{C}_{1/2+\eta}$ and then by Proposition 4.1 from \cite{gordon1999composition}, the same holds for $\Phi_{\chi}$ for every $\chi \in \mathbb{T}^{\infty}$. Consequently
\[ \Vert n^{- \Phi_{\sigma}} \Vert_{\mathcal{H}^2}^2 \leq n^{-1-2 \eta} \quad \quad \quad \quad \forall n \geq 1. \]
Now we show that $C_{\Phi}$ is Hilbert-Schmidt and then obviously bounded. The sequence $(\e_{n}^{\mu})_{n \geq 1}$ is an orthonormal basis of $\mathcal{A}_{\mu}^2$ and $C_{\Phi}(\e_n^{\mu}) \in \mathcal{H}^{\infty}$ for every $n \geq 1$: indeed $\e_n^{\mu} \in \mathcal{P}$ and $\Phi( \mathbb{C}_+) \subset \mathbb{C}_+$. Then by \cite{bailleullefevre2013}, Th.6(ii):
\[ \sum_{n=1}^{+ \infty} \Vert C_{\Phi}(\e_{n}^{\mu}) \Vert_{ \mathcal{A}_{\mu}^2}^2 =  \sum_{n=1}^{+ \infty} \int_{0}^{+ \infty} \Vert (C_{\Phi}(\e_n^{\mu}))_{\sigma} \Vert_{\mathcal{H}^{2}}^2 d\mu(\sigma)\]
\[ =  \sum_{n=1}^{+ \infty} \int_{0}^{+ \infty} \frac{\Vert n^{-\Phi_{\sigma}}  \Vert_{\mathcal{H}^{2}}^2}{w_h(n)}  d\mu(\sigma) \leq \sum_{n=1}^{+ \infty} \int_{0}^{+ \infty} \frac{n^{-1-2 \eta}}{w_h(n)}  d\mu(\sigma) \]
\[ = \sum_{n=1}^{+ \infty} \frac{n^{-1-2\eta}}{w_h(n)}.\]
Now in \cite{mccarthy2004hilbert}, it is shown that the weight $(w_h(n))$ decreases more slowly than any negative power of $n$ so there exists $C>0$ such that $w_h(n) \geq Cn^{- \eta}$ for every $n \geq 1$, then
\[  \sum_{n=1}^{+ \infty} \Vert C_{\Phi}(\e_{n}^{\mu}) \Vert_{ \mathcal{A}_{\mu}^2}^2 \leq \frac{1}{C} \sum_{n=1}^{+\infty} \frac{1}{n^{1+\eta}} < +\infty \]
and so $C_{\Phi}$ is Hilbert-Schmidt.

$(iii)$ Assume that $C_{\Phi}$ is bounded on $\mathcal{A}_{\mu}^2$. Let $P$ be a Dirichlet polynomial, we get
\[ \Vert P \circ \Phi \Vert_{\mathcal{A}_{\mu}^{2k}} = \Vert P^k \circ \Phi \Vert_{\mathcal{A}_{\mu}^2}^{1/k}  \leq \Vert C_{\Phi} \Vert^{1/k} \Vert P^k \Vert_{\mathcal{A}_{\mu}^2}^{1/k} = \Vert C_{\Phi} \Vert^{1/k} \Vert P \Vert_{ \mathcal{A}_{\mu}^{2k}}. \]
By density of the polynomials and boundedness of the point evaluation we obtain the result (as in the proof of Theorem \ref{Continuitécoplusgrandque1}).
\end{proof}
 
\begin{qu}
$\hbox{  }$

Is it true that $C_{\Phi}$ is actually nuclear?

Can we choose $\eta=0$ in $(ii)$ to get boundedness? 
\end{qu} 

\section{Compactness and Nevanlinna counting function} 
Let $p \geq 1$ and $X= \mathcal{H}^p$ or $\mathcal{A}_{\mu}^p$. We begin this section with the following criterion of compactness:

\begin{proposition}
Let $\Phi : \mathbb{C}_{1/2} \rightarrow \mathbb{C}_{1/2}$ be an analytic function which induces a bounded composition operator on $X$. Then $C_{\Phi}$ is compact on $X$ if and only if for every bounded sequence $(f_n)$ in $X$ which converges to $0$ uniformly on $\mathbb{C}_{1/2+ \varepsilon}$ for every $\varepsilon>0$, one has $\Vert C_{\Phi}(f_n) \Vert_{X} \rightarrow 0$.
\end{proposition}

\begin{proof}
The point-evaluation at $s \in \mathbb{C}_{1/2}$ is bounded on $X$, $X$ has the Fatou-property (in the sense of \cite{lefevre2010composition}, Prop 4.8) and there exists a Montel's theorem for $\mathcal{H}^{\infty}$ (see \cite{bayart2003compact}). Then it suffices to apply the same reasoning than in the proof of Proposition 4.8 from \cite{lefevre2010composition}.
\end{proof}

\begin{rem}
With this criterion, it is easy to see that if $C_{\Phi}$ is compact on $\mathcal{H}^2$ (resp. $\mathcal{A}_{\mu}^2$) then $C_{\Phi}$ is compact on $\mathcal{H}^{2k}$ (resp. $\mathcal{A}_{\mu}^{2k}$) for any $k \geq 1$.
\end{rem}

\subsection{Sufficient conditions of compactness}

In this section, we follow ideas from \cite{bayart2003compact}: we define the generalized Nevanlinna counting function associated to the spaces $ \mathcal{A}_{\mu}^2$ in order to study the compactness of composition operators on $\mathcal{A}_{ \mu}^2$. We need some preliminaries from \cite{bailleullefevre2013}. \\

For $ \sigma>0$, we define 
\[ \beta_h ( \sigma):= \int_{0}^{ \sigma} ( \sigma-u) h(u) \, du  = \int_{0}^{ \sigma}  \int_{0}^{t} h(u) \, du dt. \]

We point out that if $h$ is continuous, the two first derivatives of $ \beta_h$ are
\[ \left\lbrace
\begin{array}{lcc}
\beta_h'( \sigma) = \int_{0}^{ \sigma} h(u)du, \\
\beta_h''( \sigma) = h( \sigma).
\end{array}  \right.
 \]
 
\begin{rem} If we denote $\beta_{\alpha}$ instead of $\beta_h$ for the spaces $\mathcal{A}_{\alpha}^2$, it is easy to see that $\beta_{\alpha}(\sigma) \approx \sigma^{ \alpha+2}$ when $\sigma$ is close to $0$.
\end{rem}

\begin{theo}["Littlewood-Paley formula" \cite{bailleullefevre2013}]
Let $\eta$ be a Borel probability measure on $ \mathbb{R}$. Then
\[ \Vert f \Vert_{ \mathcal{A}_{\mu}^2}^2 =  \vert f( + \infty) \vert^2 + 4 \int_{ \mathbb{T}^{ \infty}} \int_{ 0}^{ + \infty} \int_{ \mathbb{R}} \beta_h( \sigma) \vert f_{ \chi}^{'}( \sigma+it) \vert^2 \, d \eta(t) d \sigma d m ( \chi). \]
\end{theo}

\begin{rem}
For a Dirichlet series $f$ with expansion $(1)$, $f( + \infty)$ stands for $a_1$, the constant coefficient.
\end{rem}

\begin{nota}
We shall call a $c_0$-symbol any analytic function $\Phi : \mathbb{C}_{1/2} \rightarrow  \mathbb{C}_{1/2}$ of the form $\Phi(s)=c_0s + \varphi(s)$ where $c_0$ is a nonnegative integer and $\varphi \in \mathcal{D}$ such that $C_{\Phi}$ is bounded on $\mathcal{A}_{\mu}^2$. When $c_0 \geq 1$ it is equivalent to the fact that $\varphi$ converges uniformly on $\mathbb{C}_{\varepsilon}$ for every $\varepsilon>0$ and $\varphi( \mathbb{C}_+) \subset \mathbb{C}_+$. We point out that in the sequel, the hypothesis "$c_0 \geq 1$" will be crucial.
\end{nota}

\begin{defin}
We define the generalized Nevanlinna counting function associated to a function $\beta :(0, + \infty) \rightarrow (0, + \infty)$ by

\[
N_{\beta, \Phi}(s)=
\left\lbrace
\begin{array}{clc}
\biindice{ \sum}{ a \in \mathbb{C}_+}{ \Phi(a)=s} \beta(\Re(a)) & \mbox{if} & s \in \Phi( \mathbb{C}_+),\\
\hbox{  } 0 & \mbox{else.} & 
\end{array}\right.
\]
When $\beta( \sigma)= \sigma$, we just denote $N_{\Phi}$ this function.
\end{defin}

\begin{rems}
$\hbox{ }$ 

\begin{enumerate}[(i)]
\item
$N_{\Phi}$ is associated to the space $\mathcal{H}^2$, it was already defined in \cite{bayart2003compact}.
\item
$N_{\beta, \Phi}$ is the equivalent in our framework of Dirichlet series of the generalized Nevanlinna counting function defined in \cite{kellay2012compact} in the framework of analytic functions on the unit disk.
\end{enumerate}
\end{rems}

\begin{lemme}{\label{nevanlinnageneralisee}}
Let $\Phi$ be a $c_0$-symbol with $c_0 \geq 1$. For every $s \in \mathbb{C}_+$, we have
\[ N_{\beta_h, \Phi} (s) = \int_{0}^{ \Re(s)} N_{ \Phi_u}(s) h(u) du\]
where $\Phi_u$ is defined by $\Phi_u(s):= \Phi(s+u)$.
\end{lemme}

\begin{proof}
Let $s \in \mathbb{C}_+$.
\[ \begin{array}{ccl}
	N_{ \beta_h, \Phi}(s) & = &\displaystyle{\biindice{ \sum}{ a \in \mathbb{C}_+}{ \Phi(a)=s}} \beta_h( \Re(a))  \\ 
                          & = & \displaystyle{\biindice{ \sum}{ a \in \mathbb{C}_+}{ \Phi(a)=s} \int_{0}^{ \Re(a)}} ( \Re(a)-u) h(u) du \\
                          & = & \displaystyle{\biindice{ \sum}{ \Re(a) \leq \Re(s)}{ \Phi(a)=s} \int_{0}^{ \Re(a)}} ( \Re(a)-u) h(u) du \\
                          \end{array}\]
thanks to the "Schwarz's lemma": indeed $\Phi( \mathbb{C}_{\Re(a)}) \subset  \mathbb{C}_{\Re(a)}$ and then $\Re(a)> \Re(s)$ implies $\Phi(a) \neq s$. Now by the Fubini's theorem:
\[ N_{ \beta_h, \Phi}(s) = \int_{0}^{ \Re(s)} \biindice{ \sum}{ u< \Re(a) \leq \Re(s)}{ \Phi(a)=s} ( \Re(a)-u) h(u) \, du. \]
Now it suffices to remark that 
\[  \begin{array}{ccl}
\biindice{ \sum}{ u< \Re(a) \leq \Re(s)}{ \Phi(a)=s} (\Re(a)-u) & = & \biindice{ \sum}{ u< \Re(a)}{ \Phi_u(a-u)=s} ( \Re(a)-u) \\
  & = & \biindice{ \sum}{ 0<\Re(a')}{ \Phi_u(a')=s}  \Re(a') \\
  & = &N_{ \Phi_u}(s). \\
  \end{array}\] 
\end{proof}

\begin{proposition}{\label{Inégalitédelittlewood}}
Let $\Phi$ be a $c_0$-symbol with $c_0 \geq 1$. For every $s \in \mathbb{C}_+$, we have
\[ N_{ \beta_h ,\Phi}(s) \leq \frac{\beta_h( \Re(s))}{c_0}. \]
\end{proposition}

%We shall need the following result for the proof of the proposition.
%
%\begin{prop}{\cite{bayart2003compact}}
%Let $ \Phi: \mathbb{C}_+ \rightarrow \mathbb{C}_+$ an analytic function such that $ \Phi(s) =c_0 s + \varphi(s)$ where $c_0 \geq 1$ and $ \varphi \in \mathcal{D}$. Then
%\[ N_{ \Phi}(s) \leq \frac{\Re(s)}{c_0}. \]
%\end{prop}
\begin{proof}
By the previous lemma, we know that
\[  N_{ \beta_h, \Phi} (s) = \int_{0}^{ \Re(s)} N_{ \Phi_u}(s) h(u) du. \]
We already pointed out that $\Phi_u -u: \mathbb{C}_+ \rightarrow \mathbb{C}_+$ and then by Littlewood's inequality (\cite{bayart2003compact},Prop.3):
\[ N_{ \Phi_u -u}(s') \leq \frac{\Re(s')}{c_0} \quad \forall s' \in \mathbb{C}_+.\]
But $ N_{ \Phi_u}(s) = N_{ \Phi_u-u}(s-u)$ when $\Re(s)>u>0$. Finally 
\[ N_{ \beta_h, \Phi} (s) \leq \int_{0}^{ \Re(s)} \frac{(\Re(s)-u) h(u)}{c_0}\, du = \frac{\beta_h( \Re(s))}{c_0}. \]
\end{proof}

In order to obtain an upper bound for the essential norm of $C_{\Phi}$, we need to obtain some estimates on $N_{\beta_h, \Phi}$. We follow ideas from \cite{bayart2003compact}.

\begin{proposition}{\label{estimationPhichi}}
Let $\Phi : \mathbb{C}_{+} \rightarrow \mathbb{C}_+$ and $\Phi_{k} : \mathbb{C}_{+} \rightarrow  \mathbb{C}_{+}$ where $k \geq 0$, some analytic functions. Assume that $(\Phi_k)_{k \geq 0}$ converges uniformly to $\Phi$ on every compact subset of $\mathbb{C}_+$, then for every $s \in \mathbb{C}_+$,
\[ N_{\beta_h, \Phi}(s) \leq \liminf_{k \rightarrow + \infty} N_{\beta_h,\Phi_k}(s). \]
\end{proposition}

\begin{proof}
Let us fix $s \in \Phi ( \mathbb{C}_+)$ (if $s \notin \Phi( \mathbb{C}_+)$ there is nothing to prove) and $\varepsilon>0$. Let $a_1, \, \dots , \, a_n \in \mathbb{C}_+$ satisfying $\Phi(a_i)=s$ for every $i=1, \, \dots , \, n$ and $\delta>0$ such that $2n\delta< \varepsilon$. $(\Phi_k)_{k \geq 0}$ converges uniformly to $\Phi$ on every compact set of $\mathbb{C}_+$ and so particularly on each disk $D(a_i, \delta)$ of center $a_i$ and radius $\delta$. Then by the Hurwitz's lemma, there exists $K \geq 0$ such that for any $k \geq K$, $s \in \Phi_k(D(a_i, \delta))$ for every $i=1, \, \dots , \, n$. 

Now let us fix $k \geq K$, there exist $a_1^k \in D(a_1, \delta)$, ... , $a_n^k \in D(a_n, \delta)$ such that $\Phi(a_i^k)=s$ for every $i=1, \, \dots , \, n$. By definition of $\beta_h$, we have

\[ \begin{array}{ccl}
\beta_h(\Re( a_i^k)) & = & \displaystyle{\int_{0}^{ \Re( a_i^k)} ( \Re( a_i^k)-u) h(u) du} \\
& \geq & \displaystyle{\int_{0}^{ \Re( a_i)- \delta}} ( \Re( a_i) - \delta -u) h(u) du \\
& = & \beta_h( \Re(a_i)) - \bigg{(} \delta \displaystyle{\int_{0}^{ \Re( a_i) - \delta} h(u)du} + \displaystyle{\int_{ \Re( a_i)- \delta}^{ \Re(a_i)}} ( \Re(a_i)-u) h(u) du \bigg{)} \\
& \geq & \beta_h( \Re(a_i)) - 2 \delta \\
\end{array}\]
because $ \Vert h \Vert_{1}=1$. Summing up all the terms 
\[ \beta_h( \Re(a_1^k)) + \cdots + \beta_h( \Re(a_n^k)) \geq  \beta(\Re(a_1))+ \cdots + \beta( \Re(a_n)) - 2n\delta .\]
Then for every $k \geq K$, we get
\[ \begin{array}{ccl}
\beta(\Re(a_1))+ \cdots + \beta( \Re(a_n)) & \leq  & 2n \delta + \beta_h( \Re(a_1^k)) + \cdots + \beta_h( \Re(a_n^k)) \\
& \leq &  \varepsilon + N_{\Phi_k}(s). \\
\end{array} \]
So
\[ \beta(\Re(a_1))+ \cdots + \beta( \Re(a_n)) \leq \varepsilon + \liminf_{k \rightarrow + \infty} N_{\Phi_k}(s) \]
and since $n$ and $\varepsilon$ are arbitrary, the result is proved.
\end{proof}

\begin{cor}
Let $\Phi$ a $c_0$-symbol with $c_0 \geq 1$. Assume that
\[ \sup_{ \Re(s)>0} \frac{N_{\beta, \Phi}(s)}{\beta(\Re(s))} = C < + \infty, \]
then
\[ \sup_{\chi \in \mathbb{T}^{\infty}} \sup_{ \Re(s)>0} \frac{N_{\beta, \Phi_{\chi}}(s)}{\beta(\Re(s))} = C.\]
\end{cor}

\begin{proof}
For every $\tau \in \mathbb{R}$, it is easy to see that $N_{ \beta_h, \Phi_{\tau}}(s) = N_{ \beta_h, \Phi}(s+ic_0 \tau)$ and so for every $\tau \in \mathbb{R}$ we have
\[ \sup_{ \Re(s)>0} \frac{N_{\beta, \Phi_{\tau}}(s)}{\beta(\Re(s))} = C. \]
Now for $\chi \in \mathbb{T}^{\infty}$, $\Phi_{\chi}$ is a vertical limit of a sequence $(\Phi_{\tau_n})_{n \geq 0}$ where $(\tau_n)_{n \geq 0} \subset \mathbb{R}$ and then it suffices to apply Proposition \ref{estimationPhichi} to obtain that 
\[ \sup_{ \Re(s)>0} \frac{N_{\beta, \Phi_{\chi}}(s)}{\beta(\Re(s))} \leq C \]
The other inequality is true because $\Phi=\Phi_{\chi}$ with $\chi=(1, 1, \dots)$.
\end{proof}

%\begin{proposition}{\label{Lesupnedependpasdechi}}
%Assume that
%\[ \sup_{ \Re(s)>0} \frac{N_{\beta, \Phi}(s)}{\beta(\Re(s))} = C < + \infty. \]
%Then for every $ \chi \in \mathbb{T}^{ \infty}$, the same holds for $\Phi_{\chi}$ instead of $\Phi$.
%\end{proposition}
%
\begin{rem}
In the previous corollary, we can take the supremum on a band of the form $\lbrace s \in \mathbb{C}_+, \Re(s)< \theta \rbrace$ where $\theta \in \mathbb{R}$.
\end{rem}

Let $H$ a Hilbert space and $ \mathcal{K}(H)$ the space of all compact operators on $H$. Let $T$ be a bounded operator on $H$, we recall that its essential norm is defined by 
\[ \Vert T \Vert_{H,e} = \inf \lbrace \Vert T-K \Vert, \, K \in \mathcal{K}(H) \rbrace. \]
Clearly, the essential norm of $T$ is zero if and only if $T \in \mathcal{K}(H)$. The next result will be very useful.

\begin{prop}{(\cite{shapiro1987essential},Prop5.1)}
Let $T$ be a bounded operator on $H$ and $(K_n)_{n \geq 1}$ a sequence of self-adjoint compact operators on $H$. Let $R_n= I-K_n$. Assume that $\Vert R_n \Vert=1$ for every $n \geq 1$ and that $(R_n)_{n \geq 1}$ converges pointwise to zero. Then
\[ \Vert T \Vert_{H,e} = \lim_{n \rightarrow + \infty} \Vert T R_n \Vert.\]
\end{prop}

\begin{theoreme}{\label{Thprincipal}}
Let $\Phi$ a $c_0$-symbol with $c_0 \geq 1$. Assume that $Im( \varphi)$ is bounded on $ \mathbb{C}_+$, then 
\[ \Vert C_{ \varphi} \Vert_{\mathcal{A}_{\mu}^2,e} \leq \bigg{(}2 \sup_{ \mathbb{C}_+} \vert Im( \varphi) \vert+c_0 \bigg{)} \cdot \limsup_{ \Re(s) \rightarrow 0} \frac{N_{\beta_h, \Phi}(s)}{\beta_h(Re(s))} \]
and
\[ \Vert C_{ \varphi} \Vert_{\mathcal{H}^2,e} \leq \bigg{(}2 \sup_{ \mathbb{C}_+} \vert Im( \varphi) \vert+ c_0 \bigg{)} \cdot \limsup_{ \Re(s) \rightarrow 0} \frac{N_{ \Phi}(s)}{\Re(s)}. \]
\end{theoreme}

\begin{cor}{\label{Corprincipal}}
Assume that $Im( \varphi)$ is bounded on $ \mathbb{C}_+$ then:

\begin{enumerate}[(i)]
\item
If $N_{\beta_h,\Phi}(s) = o(\beta(\Re(s))$ when $\Re(s) \rightarrow 0$ then $C_{\Phi}$ is compact on $ \mathcal{A}_{\mu}^2$.
\item
If $N_{\Phi}(s) = o(\Re(s))$ when $\Re(s) \rightarrow 0$ then $C_{\Phi}$ is compact on $\mathcal{H}^2$.
\end{enumerate}
\end{cor}

\begin{rem}
$Cor2.(ii)$ is already known (see \cite{bayart2003compact}, Th2).
\end{rem}

\begin{proof}[Proof of Th.\ref{Thprincipal}]
We only give the proof in the case of $\mathcal{A}_{\mu}^2$, the proof in the case $ \mathcal{H}^2$ is essentially the same.
 
For $n \geq 1$ and $f \in \mathcal{A}_{\mu}^2$ which has form $(1)$ we define
\[\left\lbrace    
\begin{array}{cccc}
K_n(f) &=& \displaystyle{\sum_{k=1}^{n}} a_k e_k& \hbox{ } \\
 R_n(f)  &=& (I-K_n)(f) & = \displaystyle{\sum_{k=n+1}^{+ \infty}} a_k e_k. \end{array} \right. \]

Each $K_n$ is a self-adjoint compact operator on $\mathcal{A}_{\mu}^2$ with norm $1$ because $K_n$ is a projection on a finite linear subspace and $(R_n)$ clearly converges pointwise to zero. Then by the previous proposition

\[ \Vert C_{ \Phi} \Vert_{A_{\mu}^2,e} = \lim_{n \rightarrow + \infty} \Vert C_{ \Phi} R_n \Vert = \lim_{n \rightarrow + \infty} \bigg{\lbrace} \sup_{ \Vert f \Vert_{ \mathcal{A}_{\mu}^2} \leq 1} \Vert C_{ \Phi} (R_n(f))  \Vert_{ \mathcal{A}_{\mu}^2} \bigg{\rbrace}. \]

In the sequel, we shall work on $\Omega = (0, + \infty)\times (0, 1)$ with the area \linebreak measure: we apply the Littlewood-Paley formula for $\mathcal{A}_{\mu}^2$ with the Lebesgue measure on $(0,1)$:
\[\begin{array}{ccl}
\Vert C_{ \Phi} (R_n(f))  \Vert_{\mathcal{A}_{\mu}^2}^2 & = & \vert R_n(f)\circ \Phi(+ \infty) \vert^2 \\
                                                        & + & 4 \displaystyle{\int_{ \mathbb{T}^{ \infty}}} \iint_{\Omega} \beta_h(\Re(s)) \vert (R_n(f)\circ \Phi)_{\chi}^{'}(s) \vert^2   \, ds dm( \chi). \\
                                                        \end{array} \]

We can point out that $R_n(f) \circ \Phi( + \infty) =0$ for every $n \geq 2$ because $c_0 \geq 1$ and then $\Phi( + \infty) = + \infty$. Now by Proposition 4.3 from \cite{gordon1999composition}:

\[  \Vert C_{ \Phi} (R_n(f))  \Vert_{\mathcal{A}_{\mu}^2}^2 =   4 \int_{ \mathbb{T}^{ \infty}}  \iint_{\Omega} \beta_h(\Re(s)) \vert R_n(f)^{'}_{\chi^{c_0}}( \Phi_{ \chi}(s) ) \vert^2  \vert  \Phi_{ \chi}^{'}(s) \vert^2 \, ds dm( \chi). \]

Now we make the following "change of variables": $w = \Phi_{ \chi}(s)$, $\Phi_{\chi}$ is not necessarily injective but we use the generalized formula of change of variable which involves the generalized Nevanlinna counting function (see \cite{federer1996geometric} for example). We get:
\[\begin{array}{cl}
 & \displaystyle{\iint_{\Omega}} \beta_h(\Re(s)) \vert R_n(f)^{'}_{\chi^{c_0}}( \Phi_{ \chi}(s) ) \vert^2    \vert \Phi_{ \chi}^{'}(s) \vert^2 \, ds \\
 & \\
  =& \displaystyle{ \iint_{ \Phi_{\chi}(\Omega)}}  \vert R_n(f)^{'}_{\chi^{c_0}}(s) \vert^2   N_{\beta_h, \Phi_{ \chi}}(s) \, ds. 
  \end{array} \]

$Im( \varphi)$ is bounded then there exists $A>0$ such that $\vert Im( \varphi) \vert < A$ and the same inequality holds for $\varphi_{\chi}$ for every $\chi \in \mathbb{T}^{\infty}$ instead of $\varphi$: indeed $\varphi_{\chi}$ is a vertical limit of $\varphi$. Now, when $s \in \Omega$, we have $Im(s) \in (0,1)$ and $w= \Phi_{\chi}(s)$ verifies $Im(w) = c_0 Im(s) + Im( \varphi_{\chi}(s))$ so $Im(w) \in [-A,A+c_0]$. Finally if we denote $Q= (0, + \infty) \times (-A, A+c_0)$, we get

\[\begin{array}{cl}
 & \displaystyle{\int_{\Omega}} \beta_h(\Re(s)) \vert R_n(f)^{'}_{\chi^{c_0}}( \Phi_{ \chi}(s) ) \vert^2    \vert \Phi_{ \chi}^{'}(s) \vert^2 \, ds \\
 & \\
  \leq & \displaystyle{ \iint_{Q}}  \vert R_n(f)^{'}_{\chi^{c_0}}(s) \vert^2   N_{\beta_h, \Phi_{ \chi}}(s) \, ds. 
  \end{array} \]

%\[ \int_{0}^{1} \int_{ \mathbb{R}} \beta_h(\sigma) \vert R_n(f)^{'}_{\chi^{c_0}}( \Phi_{ \chi}(\sigma + it) ) \vert^2    \vert \Phi_{ \chi}^{'}(\sigma + it) \vert^2 \, dt d \sigma \]
%\[ \leq \int_{-A}^{A+c_0} \int_{ \mathbb{R}}  \vert R_n(f)^{'}_{\chi^{c_0}}(s ) \vert^2  \vert N_{\beta_h, \Phi_{ \chi}}(s) \, dt d \sigma. \]

Let us fix $ \theta>0$, we write $Q=Q_{\theta} \cup \tilde{Q}_{\theta}$ where $Q_{\theta}= (0, \theta) \times (-A,A+c_0)$ and $\tilde{Q}_{\theta}= (\theta, + \infty) \times (-A,A+c_0)$. We shall cut the previous integral in two parts. Let $\gamma_{ \theta}$ defined by 
\[ \gamma_{ \theta} = \sup_{ \chi \in \mathbb{T}^{ \infty}} \bigg{\lbrace} \, \sup_{ 0< Re(s) < \theta} \frac{N_{\beta_h, \Phi_{ \chi}}(s)}{\beta_h(\Re(s))} \bigg{ \rbrace } =  \sup_{ 0< Re(s) < \theta} \frac{N_{\beta_h, \Phi}(s)}{\beta_h(\Re(s))}.  \]
The second equality is true by Corollary 1. We have

\[  \begin{array}{cl}
&  \displaystyle{\int_{\mathbb{T}^{\infty}} \iint_{ Q_{\theta}}}   \vert R_n(f)^{'}_{\chi^{c_0}}(s ) \vert^2   N_{ \beta_h,\Phi_{ \chi}}(s) \, ds dm( \chi) \\ 

 \leq &\gamma_{ \theta} \,   \displaystyle{ \int_{\mathbb{T}^{\infty}} \iint_{ Q_{\theta}}}   \vert R_n(f)^{'}_{\chi^{c_0}}(s ) \vert^2  \beta_h(\Re(s)) \, ds dm( \chi).  \end{array}\]

Now we take the supremum:
\[ \sup_{ \Vert f \Vert_{ \mathcal{A}_{\mu}^2} \leq 1 } \bigg{ \lbrace}   \displaystyle{\int_{\mathbb{T}^{\infty}} \iint_{ Q_{\theta}}}    \vert R_n(f)^{'}_{\chi^{c_0}}(s ) \vert^2   N_{\beta_h, \Phi_{ \chi}}(s) \, ds dm( \chi) \bigg{ \rbrace} \]
\[ \, \, \leq \gamma_{ \theta} \sup_{ \Vert f \Vert_{ \mathcal{A}_{\mu}^2} \leq 1} \, \bigg{ \lbrace} \displaystyle{ \int_{\mathbb{T}^{\infty}} \iint_{ Q_{\theta}}}    \vert R_n(f)^{'}_{\chi^{c_0}}(s ) \vert^2  \beta_h(\Re(s)) \,ds dm( \chi)  \bigg{ \rbrace} \]
\[ \leq \gamma_{ \theta} \sup_{ \Vert f \Vert_{ \mathcal{A}_{\mu}^2} \leq 1} \bigg{ \lbrace} \displaystyle{ \int_{\mathbb{T}^{\infty}}\iint_{ Q_{\theta}}}    \vert f^{'}_{\chi^{c_0}}(s ) \vert^2  \beta_h(\Re(s)) \, ds dm( \chi)\bigg{ \rbrace}. \]

In the last inequality, we use the fact that if $f$ belongs to the unit ball of $ \mathcal{A}_{\mu}^2$, $R_n(f)$ too. Finally, by using again the Littlewood-Paley formula with the normalized Lebesgue measure on $(-A, A+c_0)$ and the fact that the map $ \chi \rightarrow \chi^{c_0}$ is measure preserving we obtain:

\[ \sup_{ \Vert f \Vert_{ \mathcal{A}_{\mu}^2} \leq 1 } \bigg{ \lbrace} 4 \int_{\mathbb{T}^{\infty}} \iint_{ Q_{\theta}} \vert R_n(f)^{'}_{\chi^{c_0}}(s ) \vert^2   N_{\beta_h, \Phi_{ \chi}}(s) \, ds dm( \chi) \bigg{ \rbrace} \leq (2A+c_0)\gamma_{\theta}. \]

Now we obtain an upper bound for the second integral. By Proposition \ref{Inégalitédelittlewood}:

\[ \begin{array}{clr}
& \displaystyle{\int_{ \mathbb{T}^{ \infty}}   \iint_{\tilde{Q_{\theta}}}}  \vert R_n(f)^{'}_{\chi^{c_0}}(s ) \vert^2   N_{\beta_h, \Phi_{ \chi}}(s) \, ds dm( \chi)& \\
 \leq & \displaystyle{\int_{ \mathbb{T}^{ \infty}}    \iint_{\tilde{Q_{\theta}}}}  \vert R_n(f)^{'}_{\chi^{c_0}}(s ) \vert^2  \frac{\beta_h( \Re(s))}{c_0} \, ds dm( \chi)& \quad \qquad (\Delta)
 \end{array}\]

Recall (see \cite{bayart2003compact}, lem2) that for every $f \in \mathcal{H}^2$ with form $(1)$ and every probability measure $\eta$ on $\mathbb{R}$ we have for every $\sigma>0$:
\[ \int_{ \mathbb{T}^{ \infty}} \int_{ \mathbb{R}} \vert f_{ \chi}^{'} ( \sigma + it) \vert^2 d \eta(t)  dm( \chi) = \sum_{k=2}^{ + \infty} \vert a_k \vert^2 k^{-2 \sigma} log^2(k), \quad \quad \forall \sigma>0. \]

Hence, we get:
\[ (\Delta) = \frac{2A+c_0}{c_0} \, \int_{ \theta}^{ + \infty}  \sum_{k=n+1}^{ + \infty} \beta_h(\sigma) \vert a_k \vert^2 k^{-2 \sigma} log^2(k) d \sigma.  \]

Let us fix $\varepsilon>0$. We claim that there exists $K=K_{\theta}>0$ such that for every $k \geq K$:
\[ \int_{ \theta}^{ + \infty} \beta_h(\sigma)  k^{-2 \sigma} log^2(k) d \sigma \leq \varepsilon w_h(k). \]

Indeed, we know by \cite{mccarthy2004hilbert} that $(w_h(k))$ is a decreasing sequence that decays more slowly than any negative power, so there exists $C=C_{\theta}>0$ such that:

\[ w_h(k) \geq C k^{- \theta} \quad \forall k \geq 1. \]
Then
\[ \int_{ \theta}^{ + \infty} \beta_h(\sigma)  \frac{k^{-2 \sigma} log^2(k)}{w_h(k)} d \sigma \leq \int_{ \theta}^{ + \infty} \sigma k^{\theta - 2 \sigma} C^{-1} log^2(k) d \sigma \]
because $\beta_h( \sigma) \leq \sigma$ by definition of $\beta_h$. Clearly the right-hand side of the inequality goes to 0 when $k$ goes to infinity. \\

Finally when $n \geq K$, we obtain
\[\frac{2A+c_0}{c_0} \, \int_{ \theta}^{ + \infty}  \sum_{k=n+1}^{ + \infty} \beta_h(\sigma) \vert a_k \vert^2 k^{-2 \sigma} log^2(k) d \sigma \leq \frac{2A+c_0}{c_0} \times \varepsilon \]
because we work with functions in $ \mathcal{A}_{\mu}^2$ with norm less than $1$. Now summing up the two integrals, we obtain that for $n \geq K$,
\[
 \Vert C_{ \Phi} R_n \Vert \leq (2A+c_0)\gamma_{\theta} +  \frac{2A+c_0}{c_0} \times \varepsilon \quad \quad \quad \quad \quad \forall n \geq K.\]
At last,
\[  \begin{array}{ccl}
\Vert C_{ \Phi} \Vert_{A_{\mu}^2,e}& = & \displaystyle{\lim_{n \rightarrow + \infty}} \Vert C_{ \Phi} R_n \Vert \\
& \leq & \displaystyle{(2A+c_0)\gamma_{\theta} +  \frac{2A+c_0}{c_0} \times \varepsilon}. \\
\end{array} \]
Now since $\varepsilon$ and $\theta$ are arbitrary, we obtain the result.
\end{proof}

\subsection{Necessary conditions of compactness}
First, we recall some facts from \cite{bailleullefevre2013}. The reproducing kernel at $s \in \mathbb{C}_{1/2}$ associated to $\mathcal{A}_{\mu}^2$ is defined for every $w \in \mathbb{C}_{1/2}$ by
\[ K_{\mu,s}(w):= \sum_{n=1}^{+ \infty} \frac{n^{- \overline{s}-w}}{w_h(n)}. \]

\begin{ex}
Let $\alpha>-1$. The reproducing kernel at $s \in \mathbb{C}_{1/2}$ associated to the space $\mathcal{A}_{\alpha}^2$ is denoted by $K_{\alpha,s}$ instead of $K_{\mu_{\alpha},s}$ and we have for every $w \in \mathbb{C}_{1/2}$,
\[ K_{\alpha,s}(w) = \sum_{n=1}^{ + \infty} (1+ \log(n))^{\alpha + 1} n^{-  \overline{s} - w}. \]
\end{ex}

We saw in the previous section that the compactness is linked with the behaviour of $\Phi$ near to the imaginary axis. Then, as in \cite{bayart2003compact}, we shall work with partial reproducing kernels which are defined on $ \mathbb{C}_+$ and not only on $\mathbb{C}_{1/2}$.

\begin{defi}
Let $s \in \mathbb{C}_{+}$ and $l \geq 1$. On $ \mathcal{A}_{\mu}^2$, we define the partial reproducing kernel of order $l$ at $s \in \mathbb{C}_{+}$ by
\[ K_{\mu,s}^{l}(w):= \biindice{\sum}{n \geq 1}{p^{+}(n) \leq p_l} \frac{n^{- \overline{s}-w}}{w_h(n)}\]
where $p^{+}(n)$ is the greatest prime divisor of $n$.
\end{defi}

We claim that these partial reproducing kernels are defined on $\mathbb{C}_+$: indeed we know by \cite{mccarthy2004hilbert} that the weight $(w_h(n))$ decreases more slowly than any negative power of $n$ so if $\varepsilon>0$, there exists $C>0$ such that $w_h(n)>Cn^{- \varepsilon}$ for every $n \geq 1$. Then 
\[  \biindice{\sum}{n \geq 1}{p^{+}(n) \leq l} \frac{n^{- \Re(s)- \Re(w)}}{w_h(n)} \leq \frac{1}{C}  \biindice{\sum}{n \geq 1}{p^{+}(n) \leq l} \frac{1}{n^{\Re(s)+ \Re(w)- \varepsilon}} = \frac{1}{C} \prod_{i=1}^{l} \big{(} 1-p_i^{ -\Re(s)- \Re(w)+ \varepsilon } \big{)}^{-1}. \]
For $s \in  \mathbb{C}_+$, this quantity is then bounded for any $w \in \mathbb{C}_{ \varepsilon}$ and any $ \varepsilon>0$ and so the claim is proved.

We shall need the following proposition which follows easily from \cite{bayart2003compact},Prop.5.

\begin{proposition}{\label{Noyauxreproduisantspartiels}}
Let $\Phi$ a $c_0$-symbol with $c_0 \geq 1$ and $l \geq 1$. Assume that $ \Phi(s)=c_0s + \displaystyle{\sum_{n=1}^{+ \infty} c_n n^{-s}}$ where $c_n=0$ when $p^{+}(n)>l$. Then $C_{ \Phi}^{*}(K_{\mu,s}^{l}) = K_{\mu, \Phi(s)}^{l}$.
\end{proposition} 

\begin{lemme}
On the unit ball $B_{\mathcal{A}_{\mu}^2}$ of $\mathcal{A}_{\mu}^2$, the weak topology is the topology of uniform convergence on every half-plane $\mathbb{C}_{1/2+ \varepsilon}$ with $ \varepsilon>0$.
\end{lemme}

\begin{proof}
Let $(f_n) \subset B_{\mathcal{A}_{\mu}^2}$ which converges weakly to $0$. Let $\varepsilon>0$, by the result on the point evaluation from \cite{bailleullefevre2013}, there exists a constant $C_{\varepsilon}$ such that 
\[ \vert f_{n}(s) \vert \leq C_{ \varepsilon} \Vert f_n \Vert \quad \hbox{ for every s } \in \mathbb{C}_{1/2+ \varepsilon} \]
and then the sequence $(f_n)$ is uniformly bounded on $ \mathbb{C}_{1/2+ \varepsilon}$. Now by the Montel's theorem for $\mathcal{H}^{\infty}$ (see \cite{bayart2002hardy},Lem18), there exists a subsequence which converges uniformly on $\mathbb{C}_{1/2+ \varepsilon}$. It suffices to show now that $0$ is the unique cluster point, but the point evaluation is bounded and $(f_n)$ converges weakly to $0$ and so $0$ is the unique cluster point. \\

Conversely let $(f_n) \subset B_{\mathcal{A}_{\mu}^2}$ which converges uniformly to $0$ on $ \mathbb{C}_{1/2+ \varepsilon}$ for every $ \varepsilon>0$. There exists at least a subsequence which converges weakly to $f \in \mathcal{A}_{ \mu}^2$. But the point evaluation is bounded for every $s \in \mathbb{C}_{1/2}$ and $(f_n)$ converges uniformly to $0$ on $ \mathbb{C}_{1/2+ \varepsilon}$ for every $ \varepsilon>0$ so $f = 0$.
\end{proof}

\begin{theoreme}
Let $\Phi$ a $c_0$-symbol with $c_0 \geq 1$ and $l \geq 1$. Assume that $ \Phi(s)=c_0s + \displaystyle{\sum_{n=1}^{+ \infty} c_n n^{-s}}$ where $c_n=0$ when $p^{+}(n)>l$. We have:
\[ \Vert C_{ \Phi} \Vert_{\mathcal{A}_{\mu}^2,e} \geq  \limsup_{ \Re(s) \rightarrow 0} \frac{\Vert K_{\mu,\Phi(s)}^l \Vert_{ \mathcal{A}_{\mu}^2}}{\Vert K_{\mu,s}^l \Vert_{ \mathcal{A}_{\mu}^2}} . \]
\end{theoreme}

\begin{proof}
Let $l \geq 1$. For $s \in \mathbb{C}_+$, we define the following family of functions:
\[ k_{s}(w):= \frac{K_{\mu,s}^{l}(w)}{\Vert K_{\mu,s}^{l} \Vert_{\mathcal{A}_{\mu}^2}}, \qquad \hbox{ where } w \in \mathbb{C}_+. \]

This family is contained in $B_{ \mathcal{A}_{\mu}^2}$ and converges uniformly to $0$ on $ \mathbb{C}_{1/2+ \varepsilon}$ for every $\varepsilon>0$: indeed we have
\[ \lim_{\Re(s) \rightarrow 0 } \Vert K_{\mu,s}^{l} \Vert_{\mathcal{A}_{\mu}^2}^2 = \lim_{\Re(s) \rightarrow 0} K_{\mu,s}^{l}(s) = + \infty \]
and for every $w \in \mathbb{C}_{1/2+ \varepsilon}$:
\[ \vert K_{\mu,s}^{l}(w) \vert \leq \biindice{\sum}{n \geq 1}{p^+(n) \leq l} \frac{n^{-1/2- \varepsilon}}{w_h(n)} < + \infty \]
because the partial reproducing kernels are defined on $ \mathbb{C}_+$. 

Let $K \in \mathcal{K}( \mathcal{A}_{\mu}^2)$, we point out that $K^* \in \mathcal{K}( \mathcal{A}_{\mu}^2)$ as well and by the previous lemma we know that $(k_{s})$ converges weakly to $0$ so $K^*( k_{s})$ is norm convergent to $0$. Then 
\[ \Vert C_{ \Phi} - K \Vert = \Vert C_{\Phi}^{*} - K^{*} \Vert \geq \Vert C_{ \Phi}^{*}(k_{s}) - K^{*}(k_{s}) \Vert_{ \mathcal{A}_{\mu}^2} \geq  \Vert C_{ \Phi}^{*}(k_{s}) \Vert_{ \mathcal{A}_{\mu}^2 }- \Vert K^{*}(k_{s}) \Vert_{ \mathcal{A}_{\mu}^2}\]
and then
\[  \Vert C_{ \Phi} - K \Vert \geq \limsup_{ \Re(s) \rightarrow 0} \Vert C_{ \Phi}^{*}(k_{s}) \Vert_{ \mathcal{A}_{\mu}^2 }.\]
Now with Proposition \ref{Noyauxreproduisantspartiels}, and the fact that the inequality is true for every $K \in \mathcal{K}( \mathcal{A}_{\mu}^2)$, we obtain the result.
\end{proof}

\begin{cor}{\label{condnecessaire}}
Let $\Phi$ a $c_0$-symbol with $c_0 \geq 1$ and $l \geq 1$. Assume that $ \Phi(s)=c_0s + \displaystyle{\sum_{n=1}^{+ \infty} c_n n^{-s}}$ where $c_n=0$ when $p^{+}(n)>l$. Assume that $C_{\Phi} \in \mathcal{K}( \mathcal{A}_{\mu}^2)$, then:
\[ \lim_{ \Re(s) \rightarrow 0} \frac{\Vert K_{\mu, \Phi(s)}^l \Vert_{ \mathcal{A}_{\mu}^2}}{\Vert K_{\mu, s}^l \Vert_{ \mathcal{A}_{\mu}^2}} = 0. \]
In particular if $C_{\Phi} \in \mathcal{A}_{\alpha}^2$ with $\alpha>-1$, it implies that
\[ \lim_{\Re(s) \rightarrow 0} \frac{\Re(s)}{\Re( \Phi(s))} = 0. \]
\end{cor}

\begin{rem}
The result concerning $\mathcal{A}_{\alpha}^2$ is already known (see \cite{bayart2003compact}).
\end{rem}
\subsection{A criterion of compactness}
In this section, we still work on $\mathcal{A}_{\mu}^2$ where $\mu$ is a probability measure such that $d \mu( \sigma) = h d \sigma$ with $h$ a positive continuous function. Recall that 
\[ \beta_h( \sigma) = \int_{0}^{ \sigma} (\sigma-u) h( u) du. \]
In the spirit of \cite{kellay2012compact} (see p12) we introduce a new condition: we say that a weight $\beta$ verifies condition $(\kappa)$ if the function $G=G_{\beta}$ defined by $G(\sigma)= \beta(\sigma)/ \sigma$ (where $\sigma>0$) is such that
\[ \lim_{\eta \rightarrow 0^+} \limsup_{\sigma \rightarrow 0^+} \frac{G(\eta \sigma)}{G(\sigma)} = 0. \]

\begin{exs}
$\hbox{ }$
\begin{enumerate}[(i)]
\item
When $h$ is a non decreasing function then $(\kappa)$ is fullfilled for $\beta_h$: indeed in this case we extend continuously $G$ at $0$ by $G(0)=0$ and it suffices to check that $G$ is a convex function but it is easy to see that
\[ G'( \sigma) = \frac{\displaystyle{\int_{0}^{ \sigma} u h(u) du}}{\sigma^2} \]
and
\[ G''( \sigma) = \frac{2\displaystyle{\int_{0}^{ \sigma}} u(h(\sigma)-h(u))du}{\sigma^3}. \] 
\item
Let $\alpha>-1$, $\beta_{\alpha}$ verifies $(\kappa)$ because $G(\sigma) \approx \sigma^{\alpha+1}$ when $\sigma$ is small enough.
\end{enumerate}
\end{exs}

\begin{proposition}{\label{Kappa}}
Let $\Phi$ a $c_0$-symbol with $c_0 \geq 1$. Assume that $\beta_h$ verifies $(\kappa)$ then if
\[ \lim_{\Re(s) \rightarrow 0} \frac{\Re(s)}{\Re(\Phi(s))} = 0 \]
then
\[  N_{\beta_h, \Phi}(s)= o(\beta_h(\Re(s))) \hbox{ when } \Re(s) \rightarrow 0, \]
hence $C_{\Phi}$ is compact on $\mathcal{A}_{\mu}^2$ if $Im( \varphi)$ is bounded.
\end{proposition}

\begin{proof}
Let $s \in \mathbb{C}_+$, we have 
\[ \begin{array}{ccl}
N_{\beta_h, \Phi}(s) & = & \displaystyle{\biindice{\sum}{a \in \mathbb{C}_+}{\Phi(a)=s} \beta_h(\Re(a))} \\
& = & \displaystyle{\biindice{\sum}{a \in \mathbb{C}_+}{\Phi(a)=s} G(\Re(a)) \Re(a)}. \\
\end{array} \]
Let $\varepsilon>0$, we know that 
\[ \lim_{\Re(s) \rightarrow 0} \frac{\Re(s)}{\Re(\Phi(s))} = 0 \]
and then by Condition $(\kappa)$, there exists $\delta>0$ such that
\[G(\Re(s)) \leq \varepsilon G(\Re( \Phi(s))) \]
when $\Re(s) < \delta$. Now, thanks to the Schwarz's lemma if $\Phi(a)=s$ and $\Re(s) < \delta$ then $\Re(a)< \delta$ too. Consequently
\[ \begin{array}{ccl}
N_{\beta_h, \Phi}(s) & = & \displaystyle{\biindice{\sum}{a \in \mathbb{C}_+}{\Phi(a)=s} G(\Re(a)) \Re(a)} \\
& \leq & \varepsilon \displaystyle{\biindice{\sum}{a \in \mathbb{C}_+}{\Phi(a)=s} G(\Re(\Phi(a))) \Re(a)} \\
& = & \varepsilon G(\Re(s)) \displaystyle{\biindice{\sum}{a \in \mathbb{C}_+}{\Phi(a)=s} \Re(a)} \\
& = & \varepsilon G(\Re(s)) N_{\Phi}(s) \\
\end{array} \]
for every $s \in \mathbb{C}_+$ such that $\Re(s) < \delta$. Now by Littlewood's inequality we get:
\[N_{\beta_h, \Phi}(s) \leq \frac{\varepsilon G(\Re(s)) \Re(s)}{c_0} = \frac{\varepsilon \beta_h(\Re(s))}{c_0} \]
when $\Re(s) < \delta$ and this proves the result.
\end{proof}

\begin{theoreme}
Let $\alpha>-1$, $\Phi$ a $c_0$-symbol with $c_0 \geq 1$ and $l \geq 1$. Assume that $ \Phi(s)=c_0s + \displaystyle{\sum_{n=1}^{+ \infty} c_n n^{-s}}$ where $c_n=0$ when $p^{+}(n)>l$ and that $Im( \varphi)$ is bounded, then the following assertions are equivalent:
\begin{enumerate}[(i)]
\item
$C_{\Phi}$ is compact on $\mathcal{A}_{\alpha}^2$.
\item
$\displaystyle{\lim_{\Re(s) \rightarrow 0} \frac{\Re(s)}{\Re(\Phi(s))} = 0.}$
\item
$N_{\alpha, \Phi}(s) = o((\Re(s))^{\alpha+2})$ when $\Re(s)$ goes to $0$.
\end{enumerate}
\end{theoreme}

\begin{proof}
It suffices to use Corollary \ref{condnecessaire}, Proposition \ref{Kappa} and Corollary \ref{Corprincipal}.
\end{proof}

\begin{rem}
We point out that the condition $(ii)$ does not depend on $\alpha$.
\end{rem}

\begin{defi}
Let $\Phi : \mathbb{C}_+ \rightarrow \mathbb{C}_+$ be an analytic function and $k \in \mathbb{N}$, we say that $\Phi$ is $k-$valent if for every $w \in \mathbb{C}_+$, there exists at most $k$ solutions to $\Phi(s)=w$. If $\Phi$ is $k-$valent for some $k \in \mathbb{N}$ then we shall say that $\Phi$ is finite-valent.
\end{defi}

\begin{theoreme}
Let $l \geq 1$ and $\Phi$ be a $c_0$-symbol with $c_0 \geq 1$ such that $\Phi(s)=c_0s + \sum_{n=1}^{+ \infty} c_n n^{-s}$ where $c_n=0$ if $p^+(n) >l$. Assume that $Im( \varphi)$ is bounded and that $\Phi$ is finite-valent, then the following assertions are equivalent:
\begin{enumerate}[(i)]
\item
$C_{\Phi}$ is compact on $\mathcal{H}^2$.
\item
$\displaystyle{\lim_{\Re(s) \rightarrow 0} \frac{\Re(s)}{\Re(\Phi(s))} = 0.}$
\item
$N_{\Phi}(s) = o(\Re(s))$ when $\Re(s)$ goes to $0$.
\end{enumerate}
\end{theoreme}

\begin{proof}
We only have to proof that $(ii) \Rightarrow (iii)$: the other implications have been proven in \cite{bayart2003compact}. We assume that $\Phi$ is $k$-valent for some $k \geq 0$. Let $\varepsilon>0$ and $s \in \mathbb{C}_+$, as in the proof of Proposition \ref{Kappa} we can show that there exists $\delta>0$ such that for every $\in \mathbb{C}_+$ with $\Re(s)< \delta$:
\[ N_{\Phi}(s) \leq \varepsilon \Re(s) \biindice{\sum}{a \in \mathbb{C}_+}{\Phi(a)=s} 1 \leq \varepsilon k \Re(s).\]

\end{proof}

\begin{rem}
Let $l \geq 1$ and $\Phi$ a $c_0$-symbol with $c_0 \geq 1$ such that $\Phi(s)=c_0s + \sum_{n=1}^{+ \infty} c_n n^{-s}$ where $c_n=0$ if $p^+(n) >l$. If $Im(\varphi)$ is bounded and $C_{\varphi}$ is compact on $\mathcal{H}^2$ then
\[ \lim_{\Re(s) \rightarrow 0} \frac{\Re(s)}{\Re(\Phi(s))} = 0 \]
and consequently $C_{\Phi}$ is compact on $\mathcal{A}_{\alpha}^2$ for any $\alpha>-1$.
\end{rem}

\begin{qu}
Is the previous implication true for general $c_0$-symbols?
\end{qu}  

\section{Compactness and Carleson measures}
In this section, we are going to obtain a sufficient condition of compactness for $C_{\Phi}$ with a "Carleson-measure" condition. We still work on $\mathcal{A}_{\mu}^2$ where $ \mu$ is a probability measure on $(0, + \infty)$ such that $0 \in Supp( \mu)$.

When $\Phi$ is a $c_0$-symbol with $c_0 \geq 1$, it is shown in \cite{bayart2002hardy} that $\Phi$ admits radial limits and so if $\lambda$ is the Lebesgue measure on $\mathbb{R}$,
\[ \Phi^{*}(it):= \lim_{ \sigma \rightarrow 0^+} \Phi( \sigma+it) \hbox{ exists } \lambda \hbox{-almost everywhere.} \]

%\begin{defi}
%Let $\xi > 0$. $\psi_{\xi}$ will be the following map 
%\begin{equation*}
%\left\lbrace
%\begin{array}{cccc}
%\psi_{\xi}:  &\mathbb{C}_+ & \rightarrow & \mathbb{D}\\
%\hbox{  }     & s           & \mapsto     & \frac{s-\xi}{s+\xi}.
%\end{array}\right.
%\end{equation*}
%\end{defi}

\begin{nota}
Let $t \in \mathbb{R}$ and $h>0$. We define the Carleson window centered at $it$ and of size $h$ by $H(t,h):= \lbrace s \in \mathbb{C}_+, \, \vert s-it \vert<h \rbrace$.
\end{nota}

\begin{defin}
$\hbox{  }$
\begin{enumerate}[(i)]
\item
We denote $\lambda_{\Phi}$ the pullback measure of $\lambda$ by $\Phi^*$, \textit{ie}, for every open set $\Omega \subset \mathbb{C}_+$:
\[ \lambda_{\Phi}(\Omega):= \lambda( \lbrace t \in \mathbb{R}, \Phi^*(it) \in \Omega \rbrace) .\]
The Carleson function associated to $\lambda_{\Phi}$ is then defined by
\[ \rho_{\Phi}(h):= \sup_{t \in \mathbb{R}} \lambda_{\Phi}(H(t,h)).\]
\item
Let $\lambda_{\mu}= \lambda \otimes \mu$. We denote by $\lambda_{\mu,\Phi}$ the pullback measure  of $\lambda_{\mu}$ by $\Phi$, \textit{ie}, for every open set $\Omega \subset \mathbb{C}_+$:
\[ \lambda_{\mu,\Phi}(\Omega):= \lambda_{\mu}( \lbrace s \in \mathbb{C}_+, \Phi(s) \in \Omega \rbrace) .\]
The Carleson function associated to $\lambda_{\Phi}$ is then defined by
\[ \rho_{\mu,\Phi}(h):= \sup_{t \in \mathbb{R}} \lambda_{\mu,\Phi}(H(t,h)).\]
\end{enumerate}
\end{defin}

In the sequel we shall denote $\beta$ instead of $\beta_h$. The next theorem is the main result of this section. 

\begin{theoreme}{\label{CompNevanCarlDir}} 
Let $\Phi$ be a $c_0$-symbol with $c_0 \geq 1$. There exists $K>0$ (independent of $\Phi$) such that: 
\begin{enumerate}[(i)]
\item
\[\sup_{s \in H(t,h/2)} N_{ \Phi}(s) \leq K \cdot \lambda_{\Phi} \big{(} H(t,2c_0h) \big{)} \quad \hbox{ for every } h \hbox{ small enough.}\]

In particular: 
\[\biindice{\sup}{s \in \mathbb{C}_+}{\Re(s)<h/2} N_{ \Phi}(s) \leq K \cdot \rho_{\Phi} \big{(} 2c_0h\big{)} \quad \hbox{ for every } h \hbox{ small enough.}\]
\item
\[\sup_{s \in H(t,h/2)} N_{\beta, \Phi}(s) \leq K \cdot \lambda_{\mu,\Phi} \big{(} H(t,2c_0h) \big{)} \quad \hbox{ for every } h \hbox{ small enough.}\]

In particular: 
\[\biindice{\sup}{s \in \mathbb{C}_+}{\Re(s)<h/2} N_{ \beta, \Phi}(s) \leq K \cdot \rho_{\mu, \Phi} \big{(} 2c_0h\big{)} \quad \hbox{ for every } h \hbox{ small enough.}\]
\end{enumerate}
\end{theoreme}

\begin{cor}{\label{Corcarleson}}
Let $\Phi$ be a $c_0$-symbol with $c_0 \geq 1$ such that $Im( \varphi)$ is bounded on $\mathbb{C}_+$. Then there exists $C>0$ such that:

\begin{enumerate}[(i)]
\item
\[ \Vert C_{\Phi} \Vert_{\mathcal{H}^2,e} \leq C \limsup_{h \rightarrow 0} \frac{\rho_{\Phi}(h)}{h}. \]
In particular if $\rho_{\Phi}(h) = o(h)$ when $h \rightarrow 0$ then  $C_{\Phi}$ is compact on $\mathcal{H}^2$.
\item
\[ \Vert C_{\Phi} \Vert_{\mathcal{A}_{\mu}^2,e} \leq C \limsup_{h \rightarrow 0} \frac{\rho_{\mu,\Phi}(h)}{\beta(h)}. \]
In particular if $\rho_{\mu,\Phi}(h) = o(\beta(h))$ when $h \rightarrow 0$ then  $C_{\Phi}$ is compact on $\mathcal{A}_{\mu}^2$.
\end{enumerate}
\end{cor}

\begin{ex}
When $\mu=\mu_{\alpha}$, we denote $N_{\alpha}$ (resp. $\rho_{\alpha, \Phi}$) the corresponding generalized Nevanlinna counting function (resp. the Carleson function). We know that in that case, $\beta_{\alpha}(\sigma) \approx \sigma^{\alpha+2}$ when $\sigma$ is small and then if $Im( \varphi)$ is bounded:
\[ \Vert C_{\Phi} \Vert_{e, \mathcal{A}_{\alpha}^2} \leq C \limsup_{h \rightarrow 0} \frac{\rho_{\alpha,\Phi}(h)}{h^{\alpha+2}}. \]
\end{ex}

\begin{proof}[Proof of Cor.\ref{Corcarleson}]
We only give the proof of $(i)$. Let $s \in \mathbb{C}_{+}$ with $\Re(s)$ small enough, we have
\[ N_{\Phi}(s) \leq \sup_{s' \in H(Im(s),2\Re(s))} N_{ \Phi}(s') \leq K \cdot \rho_{\Phi} \big{(} 8c_0\Re(s)\big{)} \]
by theorem \ref{CompNevanCarlDir}. Then we obtain
\[ \frac{ N_{\Phi}(s)}{\Re(s)} \leq 8c_0 K \cdot \frac{\rho_{\Phi} \big{(} 8c_0 \Re(s) \big{)}}{8c_0 \Re(s)}\]
and so the result.
\end{proof}

In order to prove Theorem \ref{CompNevanCarlDir}, we need to recall some facts about Carleson measures and the classical Nevanlinna counting function on the unit disk.

Let $\phi: \mathbb{D} \rightarrow \mathbb{D}$ be an analytic self-map. The Nevanlinna counting function of $\phi$ is defined for $w \in \phi( \mathbb{D})$ by
\[ N_{ \phi}(w) = \biindice{\sum}{ z \in \mathbb{D}}{ \phi(z) = w} \log(1/\vert z \vert). \]
When $w \notin \phi( \mathbb{D})$, $N_{ \phi}(w)=0$ by convention.

We denote by $m_{ \phi}$ the pullback measure of the Lebesgue measure on the unit circle by $ \phi^{*}$, the boundary values function of $ \phi$ (which exists because $\phi \in H^{\infty}( \mathbb{D})$). By abusing the notation, we still denote $\phi$ instead of $\phi^*$.

For $\eta \in \mathbb{T}$ and $h>0$, the Carleson window $S(\eta ,h)$ is defined by
\[ S( \eta,h):= \lbrace z \in \mathbb{D}, \, \vert z- \eta \vert <h \rbrace. \]

We shall make a crucial use of the following result (see \cite{lefevre2011nevanlinna},Th1.1):

Let $ \phi: \mathbb{D} \rightarrow \mathbb{D}$ an analytic self-map. For every $\beta>1$, there exists a universal constant $C_{ \beta}>0$ such that:
\[ (1/C_{\beta}) m_{ \phi}(S( \xi, h)) \leq \sup_{z \in S( \xi,h)} N_{ \phi}(z) \leq C_{ \beta} m_{ \phi}(S( \xi, \beta h)) \quad \quad \quad \circledast \circledast\]
for every $0<h< (1- \vert \phi(0) \vert)/ \beta$ and every $\xi \in \mathbb{T}$.

\begin{proof}[Proof of Th.\ref{CompNevanCarlDir}(i)] $\hbox{  }$

\noindent $\rhd$ Step one: We prove the result when $t=0$.

We shall use a family of functions $(\psi_{\xi})_{\xi>0}$ where $\psi_{\xi}$ is defined by
 \begin{equation*}
\left\lbrace
\begin{array}{cccc}
\psi_{\xi}:  &\mathbb{C}_+ & \longrightarrow & \mathbb{D}\\
\hbox{  }     & s           & \longmapsto     & \displaystyle{\frac{s-\xi}{s+\xi}}.
\end{array}\right.
\end{equation*}
%Let $h_0>0$ and $s \in H(0,h_0)$. Let $1/2>h_1>0$ such that that $m_{\Phi}(H(0,h_1)) < + \infty$ (if it does not exist, there is nothing to prove).

\begin{claim}
For every $h>0$, $H(0,h/2) \subset \psi_{c_0 \xi}^{-1}( S( -1, h/ \xi))$.

Let us prove this fact:
\[ \begin{array}{ccccc}
s \in  \psi_{c_0 \xi}^{-1}( S( -1, h/ \xi)) & \Leftrightarrow &\psi_{c_0 \xi}(s) \in S( -1, h/ \xi) & \Leftrightarrow & \displaystyle{\bigg{ \vert } \frac{s-c_0 \xi}{s+c_0 \xi} + 1 \bigg{ \vert} < \frac{h}{\xi}} \\
&&&& \\ 
& \Leftrightarrow & \displaystyle{\bigg{ \vert } \frac{2s}{s+c_0 \xi}\bigg{ \vert} < \frac{h}{\xi}} &  \Leftrightarrow & \vert s \vert < \displaystyle{\frac{h}{2} \bigg{\vert} \frac{s+c_0\xi}{\xi} \bigg{ \vert}}. \\ 
\end{array} \]
Now if $s \in H(0,h/2)$ then $\vert s \vert < \frac{h}{2}$ and it suffices to remark that for every $\xi>0$,
\[ \frac{\vert s+c_0 \xi \vert}{\xi} \geq \frac{\Re(s+c_0 \xi)}{\xi} \geq c_0 \geq 1,\]
and this justifies the claim.
\end{claim}
 
Let $h>0$ and $s \in H(0,h/2)$. Then $w_{\xi}:= \psi_{c_0 \xi}(s) \in S(-1,h/\xi)$. Assume that $s \in \Phi( \mathbb{C}_+)$ (if it is not the case, there is nothing to prove) and let $s_1, \dots , s_n \in \mathbb{C}_+$ such that $ \Phi(s_i) = s$ for every $i=1, \dots , n$. \\

For $\xi >0$, we define the analytic self-map $ \Theta_{ \xi}: \mathbb{D} \rightarrow \mathbb{D}$ by $ \Theta_{\xi} = \psi_{c_0 \xi} \circ \Phi \circ \psi_{\xi}^{-1}$. We note that $\Theta_{\xi}( \psi_{\xi}(s_i)) = \psi_{c_0 \xi} ( \Phi(s_i)) = \psi_{c_0 \xi}(s)=w_{\xi}$ and so
\[ \sum_{i=1}^{n} \log(1/ \vert \psi_{\xi}(s_i) \vert) \leq N_{ \Theta_{\xi}}(w_{\xi}). \]
For $i \in \lbrace 1, \dots , n \rbrace$ we have 
\[ \log(1/ \vert \psi_{\xi}(s_i) \vert) = \frac{1}{2} \log \bigg{ \vert } \frac{\vert s_i \vert^2 + 2 \xi \Re(s_i) + \xi^2}{\vert s_i \vert^2 - 2 \xi \Re(s_i) + \xi^2} \bigg{ \vert} = \frac{1}{2} \log \bigg{ \vert } 1+ \frac{4\xi \Re(s_i)}{\vert s_i \vert^2 - 2 \xi \Re(s_i) + \xi^2} \bigg{ \vert}. \]
Let $\varepsilon> 0$. For $\xi$ large enough we obtain that for every $i \in \lbrace 1, \dots , n \rbrace$:
\[ 2(1- \varepsilon) \frac{\Re(s_i)}{\xi} \leq \log(1/ \vert \psi_{\xi}(s_i) \vert). \]
And so 
\[ 2(1- \varepsilon) \sum_{i=1}^{n} \frac{\Re(s_i)}{\xi} \leq N_{ \Theta_{\xi}}(w_{\xi}). \]
Now using $\circledast \circledast$ with $\beta=2$ and remembering that $w_{\xi} \in S(-1, h/ \xi)$ for every $\xi>0$, we obtain
\[ N_{ \Theta_{\xi}}(w_{\xi}) \leq C_{2} m_{ \Theta_{\xi}}(S(-1, 2h/\xi)) \]
for every $0<h/\xi< (1- \vert \Theta_{\xi}(0) \vert)/2$. Point out that 
\[ \Theta_{\xi}(0) = \psi_{c_0 \xi}( \Phi( \xi)) = \frac{\varphi(\xi)}{2c_0\xi+ \varphi(\xi)} \]
and this quantity goes to $0$ when $\xi$ goes to infinity. So we define
\[ m_0 = \inf_{ \xi > 1} \xi(1- \vert \Theta_{\xi}(0) \vert)/2 >0. \]
It is clear that $m_0 \leq 1/2$. Let us point out that $m_0=m_0( \varphi) = m_0( \varphi(\cdot + i \tau))$ for any $\tau \in \mathbb{R}$ (this remark will be useful in step $2$). When $h < m_0$, we get
\[ N_{ \Theta_{\xi}}(w_{\xi}) \leq C_{2} m_{ \Theta_{\xi}}(S(-1, 2h/\xi)) \]
and
\[ \begin{array}{ccl}
m_{ \Theta_{\xi}}(S(-1, 2h/ \xi)) &= & \displaystyle{m( \lbrace \eta \in \mathbb{T}, \Theta_{\xi}(\eta) \in S(-1, 2h/ \xi) \rbrace)} \\
& = & \displaystyle{m( \lbrace \eta \in \mathbb{T}, \Psi_{c_0 \xi} \circ \Phi^* \circ \psi_{\xi}^{-1}(\eta) \in S(-1, 2h/ \xi) \rbrace)} \\
&= &  \displaystyle{\int_{ \lbrace \Psi_{c_0 \xi} \circ \Phi^* \circ \psi_{\xi}^{-1}(e^{i \theta}) \in S(-1, 2h/ \xi) \rbrace)} \, \,  \, \frac{d\theta}{2 \pi}}. \\
\end{array}\]
%We note that since $\varphi$ is bounded on $\mathbb{C}_1$, we can choose the same $m_0$ for every translate $\varphi(it'+ \cdot)$ of $\varphi$ and then the first part of this proof (regarding $\widetilde{\Phi}$) is justify. 
Now we make the following change of variables $iu = \psi_{ \xi}^{-1}(e^{i \theta})$:
\[ \begin{array}{ccl}
m_{ \Theta_{\xi}}(S(-1, 2h/ \xi)) & = &\displaystyle{\int_{ \lbrace \Phi^*(iu) \in \Psi_{c_0 \xi}^{-1} (S(-1, 2h/ \xi)) \rbrace} \, \, \frac{\xi}{\pi( \xi^2+u^2)} du} \\
& = &\displaystyle{\int_{ \lbrace \vert \Phi^*(iu) \vert < h \frac{\vert \Phi^*(iu) + c_0 \xi \vert}{\xi} \rbrace} \, \, \frac{\xi}{\pi( \xi^2+u^2)} du}. \\
\end{array} \]
Then for $\xi > 1$ large enough we obtain
\[ 2(1- \varepsilon) \sum_{i=1}^{n} \frac{Re(s_i)}{\xi} \leq C_2 \int_{ \lbrace \vert \Phi^{*}(iu) \vert < h \frac{\vert \Phi^*(iu) + c_0 \xi \vert}{\xi} \rbrace} \, \, \frac{\xi}{\pi( \xi^2+u^2)} du \]
and then 
\[  2(1- \varepsilon) \sum_{i=1}^{n} \Re(s_i) \leq C_2 \int_{ \lbrace \vert \Phi^{*}(iu) \vert < h \frac{\vert \Phi^{*}(iu) + c_0 \xi \vert}{\xi} \rbrace}  \, \, \frac{du}{\pi}. \]
But if $\vert \Phi^*(iu) \vert < h \frac{\vert \Phi^*(iu) + c_0 \xi \vert}{\xi}$ then $ \vert \Phi^{*}(iu) \vert < \frac{c_0h}{(1-h/\xi)} < \frac{c_0h}{(1-h)} <2c_0h$ since $h<m_0 \leq 1/2$. Then
\[ 2(1- \varepsilon) \sum_{i=1}^{n} Re(s_i) \leq C_2 m_{\Phi}(H(0,2c_0h)). \]
Since $n$ is arbitrary we get:
\[ 2(1- \varepsilon) N_{\Phi}(s) \leq C_2 m_{\Phi}(H(0,2c_0h)). \]
for any $\varepsilon>0$ and the first step is proved. \\

$\rhd$ Step two: we prove now the result when $t \neq 0$.

We have
\[ \sup_{s \in H(c_0t,h/2)} N_{ \Phi}(s) = \sup_{s \in H(0,h/2)} N_{ \Phi}(s+ic_0t). \]
But if $\tilde{\Phi}(s):= c_0s + \varphi(s+it)=\Phi(s+it)-ic_ot$, we obtain for $s \in H(0,h/2)$:
\[\begin{array}{cclcl}
N_{ \Phi}(s+ic_0t) &=& \displaystyle{ \biindice{\sum}{a \in \mathbb{C}_+}{\Phi(a)=s+ic_0t} \Re(a)} 
& = &\biindice{\sum}{a \in \mathbb{C}_+}{c_0(a-it) + \varphi(it+ a-it)=s} \Re(a-it) \\
& = & \biindice{\sum}{a' \in \mathbb{C}_+}{\tilde{\Phi}(a')=s} \Re(a')& =& N_{\tilde{\Phi}}(s). \\
\end{array}\]
Applying step one to $\tilde{\Phi}$, we get
\[ \sup_{s \in H(0,h/2)} N_{ \Phi}(s+ic_0t) = \sup_{s \in H(0,h/2)} N_{ \tilde{\Phi}}(s) \leq K \lambda_{\tilde{\Phi}}(H(0,2c_0h))\]
for every $h$ small enough. Now it suffices to remark that
\[ \begin{array}{ccl}
\lambda_{\tilde{\Phi}}(H(0,2c_0h)) &= & \displaystyle{\lambda( \lbrace t' \in \mathbb{R}, \, \tilde{\Phi}^{*}(it') \in H(0,2c_0h) \rbrace) }\\
&=& \displaystyle{ \lambda( \lbrace t' \in \mathbb{R}, \, \vert \Phi^{*}(it+it') - ic_0t \vert<2c_0h \rbrace)}  \\
& = & \displaystyle{\lambda( \lbrace t'' \in \mathbb{R}, \, \vert \Phi^{*}(it'') - ic_0t \vert<2c_0h \rbrace)}\\
& = &\lambda_{\Phi}(H(c_0t,2c_0h)). \\
\end{array}\]

\end{proof}

\begin{proof}[Proof of Th.\ref{CompNevanCarlDir}(ii)]
By Lemma \ref{nevanlinnageneralisee}, we know that
\[ N_{\beta_h, \Phi}(s) = \int_{0}^{\Re(s)} N_{\Phi_u}(s) h(u) du. \]
Now in the proof of $(i)$, we point out that we can choose the same $m_0$ for $\Phi_{u}$ for every $u \geq 0$ and so for $s \in H(t,\tilde{h}/2)$ with $\tilde{h}<m_0$ we have
\[ N_{\beta_h, \Phi}(s)  \leq K \int_{0}^{\Re(s)} \lambda_{\Phi_u}(H(t,2c_0 \tilde{h})) h(u) du \]
\[ \leq K \int_{0}^{+ \infty} \lambda ( \lbrace t' \in \mathbb{R}, \Phi(u+it') \in H(t,2c_0 \tilde{h}) \rbrace) h(u) du = K \lambda_{\mu, \Phi}(H(t,2c_0 \tilde{h}) \]
thanks to the Fubini's theorem.
\end{proof}

\begin{small}
\noindent \textbf{Acknowledgements.} I would like to thank Pascal Lef\`evre for his advice regarding this paper. 

\end{small}

\begin{footnotesize}
\nocite{*}
\bibliography{biblio2}
\bibliographystyle{plain}
\end{footnotesize}

{\small 
\noindent{\it 
Univ Lille-Nord-de-France UArtois, \\ 
Laboratoire de Math\'ematiques de Lens EA~2462, \\
F\'ed\'eration CNRS Nord-Pas-de-Calais FR~2956, \\
F-62\kern 1mm 300 LENS, FRANCE \\
maxime.bailleul@euler.univ-artois.fr
}}

\end{document}